\documentclass[11pt,reqno]{amsart}
\usepackage{graphicx} 
\usepackage{amsfonts,amssymb,amsmath,amsthm,adjustbox}
\usepackage[margin=1in]{geometry}
\usepackage[hidelinks]{hyperref}
\usepackage{enumerate}
\usepackage{microtype}
\usepackage{mathtools}
\usepackage{cite}
\usepackage{color}
\usepackage{tikz-cd}
\usepackage{float}
\usepackage{comment}

\theoremstyle{plain}
\newtheorem{theorem}{Theorem}
\newtheorem{corollary}[theorem]{Corollary}
\newtheorem{lemma}[theorem]{Lemma}
\newtheorem{proposition}[theorem]{Proposition}
\newtheorem{conjecture}[theorem]{Conjecture}
\theoremstyle{definition}
\newtheorem{definition}[theorem]{Definition}

\usepackage{tikz}
\usetikzlibrary{fit,shapes.geometric}

\newcommand{\none}[1]{#1}

\newcommand{\tri}[1]{%
\tikz[baseline=(X.base)]{
\node (X) {$#1$};
\node[draw, regular polygon, regular polygon sides=3, fit=(X), inner sep=2pt] {};
}}

\newcommand{\pent}[1]{%
\tikz[baseline=(X.base)]{
\node (X) {$#1$};
\node[draw, regular polygon, regular polygon sides=5, fit=(X), inner sep=2pt] {};
}}

\newcommand{\both}[1]{%
\tikz[baseline=(X.base)]{
\node (X) {$#1$};
\node[draw, regular polygon, regular polygon sides=5, fit=(X), inner sep=2pt] {};
\node[draw, regular polygon, regular polygon sides=3, fit=(X), inner sep=2pt] {};
}}

\theoremstyle{remark}

\newcommand{\mc}[1]{\href{https://beta.lmfdb.org/ModularCurve/Q/#1/}{\texttt{#1}}}
\newcommand{\ec}[1]{\href{https://www.lmfdb.org/EllipticCurve/Q/#1/}{\texttt{#1}}}

\def\Q{\mathbb{Q}}

\def\Z{\mathbb{Z}}
\def\F{\mathbb{F}}

\def\Gal{\operatorname{Gal}}
\def\lcm{\operatorname{lcm}}

\DeclareMathOperator{\rad}{rad}

\DeclareMathOperator{\GL}{GL}
\DeclareMathOperator{\SL}{SL}
\DeclareMathOperator{\PSL}{PSL}

\DeclareMathOperator{\tr}{tr}

\DeclareMathOperator{\Frob}{Frob}
\DeclareMathOperator{\Aut}{Aut}

\newcommand{\isom}{\simeq}

\DeclareMathOperator{\Prime}{prime}

\title{Congruence obstructions in the refined Koblitz conjecture}

\author{Sung Min Lee}
\address{Sung Min Lee, Department of Mathematics, Wake Forest University, Winston-Salem, NC}

\author{Jacob Mayle}
\address{Jacob Mayle, Department of Mathematical Sciences, University of Delaware, Newark, DE}

\author{Rakvi}
\address{Rakvi, Department of Mathematics, The University of Maine, Orono, ME}


\begin{document}

\begin{abstract}
Let $E/\Q$ be an elliptic curve. Zywina's refinement of Koblitz's conjecture predicts that there are infinitely many primes $p$ of good reduction for which $\#E_p(\F_p)$ is prime precisely when $E$ has no congruence obstruction. For non-CM elliptic curves over $\Q$, we classify all primitive congruence obstructions. In particular, we show that a primitive obstruction of composite level can occur only at $6$, $10$, $14$, $15$, or $30$ and determine the corresponding Galois images. As a consequence, for non-CM elliptic curves, the existence of any congruence obstruction is detected by the mod $210$ Galois image, so the positivity of the Koblitz--Zywina constant is determined at level $210$. We also determine which primitive obstruction levels can occur simultaneously. Finally, we prove, assuming GRH, that if $E$ has no congruence obstruction, then for every $0<\kappa<1/8$ there are infinitely many primes $p$ of good reduction for which the least prime divisor of $\#E_p(\F_p)$ is greater than $\kappa\log p$.
\end{abstract}

\maketitle

\section{Introduction} \label{S:Intro}

Let $E/\Q$ be an elliptic curve, and let $N_E$ denote its conductor. For each prime $p\nmid N_E$, the reduction $E_p$ of $E$ modulo $p$ is an elliptic curve over $\F_p$, and
\[
\#E_p(\F_p)=p+1-a_p(E),
\]
where $a_p(E)$ is the trace of Frobenius. Motivated in part by applications to cryptography, Koblitz \cite{MR917870} studied the frequency with which $\#E_p(\F_p)$ is prime as $p$ varies. Building on the heuristics of Hardy and Littlewood \cite{MR1555183}, he conjectured an asymptotic formula for the function
\[
\pi_E^{\Prime}(x)
\coloneqq
\#\{p\leq x:p\nmid N_E\text{ and }\#E_p(\F_p)\text{ is prime}\}.
\]
In particular, the conjecture predicts that $\pi_E^{\Prime}(x)\to\infty$ as $x\to\infty$ whenever $E$ is not rationally isogenous to an elliptic curve over $\Q$ with nontrivial rational torsion \cite[Conjectures~A and~B]{MR917870}.

This conjecture stood for two decades, until Jones discovered a striking counterexample, discussed in \cite[Section~1.1]{MR2805578}. The elliptic curve
\[
E\colon y^2=x^3+9x+18
\]
is not rationally isogenous to an elliptic curve with nontrivial rational torsion. Nevertheless, $\#E_p(\F_p)$ is composite for every prime $p \neq 5$ of good reduction. Indeed, if $p\equiv 1\pmod 4$, then $3$ divides $\#E_p(\F_p)$, while if $p\equiv 3\pmod 4$, then $2$ divides $\#E_p(\F_p)$. The source of this phenomenon is an entanglement between the $2$- and $3$-division fields of $E$: both contain $\Q(i)$ as a subfield.

Following Jones's counterexample, Zywina \cite{MR2805578} refined Koblitz's conjecture to account for entanglements among division fields. A key notion in Zywina's refinement is that of a congruence obstruction, which we now recall.

\begin{definition} \label{D:CongObs}
Let $E/\mathbb{Q}$ be an elliptic curve, and let $m \geq 2$ be an integer. We say that $E$ has a \emph{congruence obstruction modulo $m$} if, for all but finitely many primes $p$ of good reduction,
\[
\gcd(\#E_p(\mathbb{F}_p), m) > 1.
\]
If, moreover, $E$ has no congruence obstruction modulo any proper divisor $n \geq 2$ of $m$, then we say that the congruence obstruction modulo $m$ is \emph{primitive}.
\end{definition}

Congruence obstructions admit an equivalent formulation in terms of Galois images, which we recall in Section~\ref{S:Prelims}. Zywina formulates both quantitative and qualitative versions of the refined conjecture for a fixed positive integer $t$, asking how often $\#E_p(\F_p)/t$ is prime. For $t=1$, the qualitative conjecture asserts the following \cite[Conjecture~2.2]{MR2805578}.

\begin{conjecture}
Let $E/\Q$ be an elliptic curve. If $E$ has no congruence obstruction modulo any integer $m \geq 2$, then there are infinitely many primes $p$ of good reduction for $E$ such that $\#E_p(\F_p)$ is prime.
\end{conjecture}

The main objective of this article is to classify these obstructions for non-CM elliptic curves over $\Q$. For an integer $m \geq 2$, let $\operatorname{rad}(m) = \prod_{p \mid m} p$ denote the radical of $m$. Since
\[
\gcd(\#E_p(\F_p),m)>1
\quad\iff\quad
\gcd(\#E_p(\F_p),\operatorname{rad}(m))>1,
\]
an elliptic curve has a congruence obstruction modulo $m$ if and only if it has one modulo $\operatorname{rad}(m)$. Consequently, every primitive obstruction level is squarefree.

At prime level, the obstruction is governed by rational torsion in the isogeny class. Indeed, a theorem of Katz \cite[Theorem 2]{MR604840} implies that $E$ has a congruence obstruction modulo a prime $\ell$ if and only if $E$ is rationally isogenous to an elliptic curve with a rational point of order $\ell$. Together with Mazur's theorem on rational torsion \cite{MR488287}, this shows that a prime level congruence obstruction can occur only for
\(
\ell\in\{2,3,5,7\};
\)
see Proposition~\ref{P:prime-level}.

For composite levels, the situation is more subtle because of the role of entanglements among the division fields. Jones's counterexample has a primitive congruence obstruction modulo $6$ arising from an entanglement between its $2$- and $3$-division fields. Our main theorem gives a complete classification of primitive congruence obstructions at arbitrary composite levels.

\begin{theorem} \label{T:class} Let $E/\Q$ be a non-CM elliptic curve, and let $m \geq 2$ be a composite integer. Then $E$ has a primitive congruence obstruction modulo $m$ if and only if
\[
m \in \{ 6, 10, 14, 15, 30 \},
\]
and, up to conjugacy, the mod $m$ Galois image of $E$ is one of the groups listed in Tables~\ref{tab:infpoints} and \ref{tab:exceptional-images}. 
\end{theorem}

The proof reveals a sharp difference between the levels $6$, $10$, and $14$ as compared to levels $15$ and $30$ for congruence obstructions. For each $m\in\{6,10,14\}$, there are infinitely many $\overline{\Q}$-isomorphism classes containing an elliptic curve over $\Q$ with a primitive congruence obstruction modulo $m$. By contrast, the non-CM elliptic curves over $\Q$ with a primitive congruence obstruction modulo $15$ are precisely
\[
\ec{450.c1},\quad
\ec{450.c2},\quad
\ec{450.c3},\quad
\ec{450.c4},
\]
and those with a primitive congruence obstruction modulo $30$ are precisely
\[
\ec{14400.ef1},\quad
\ec{14400.ef2},\quad
\ec{14400.ef3},\quad
\ec{14400.ef4}.
\]
The isogeny class \(\ec{14400.ef}\) is the quadratic twist by \(-2\) of the isogeny class \(\ec{450.c}\). Here and below, we use the labels from the $L$-Functions and Modular Forms Database (LMFDB) \cite{lmfdb}.

It is possible that a single elliptic curve has primitive obstructions modulo more than one integer. For example, the elliptic curve with LMFDB label \ec{66.c3} has a rational point of order $10$ and therefore a primitive congruence obstruction modulo both $2$ and $5$. Building on Theorem~\ref{T:class}, we also determine which distinct primitive obstruction levels can occur for the same elliptic curve. We define the \emph{primitive obstruction set} of $E$ to be
\[
\operatorname{PrimObs}(E)
\coloneqq
\{m\geq 2:E\text{ has a primitive congruence obstruction modulo }m\}.
\]
Specifically, we prove the following result.

\begin{theorem} \label{T:set}
If $E/\Q$ is a non-CM elliptic curve, then
\[
\operatorname{PrimObs}(E)\in
\left\{
\emptyset,\,
\{2\},\,
\{3\},\,
\{5\},\,
\{6\},\,
\{7\},\,
\{10\},\,
\{14\},\,
\{15\},\,
\{30\},\,
\{2,3\},\,
\{2,5\}
\right\}.
\]
Moreover, every set in the displayed list occurs.
\end{theorem}
In particular, a primitive obstruction of composite level never occurs together with any other primitive obstruction. We see that the possible primitive obstruction levels are
\begin{equation} \label{E:prim-obs-levels}
2,\ 3,\ 5,\ 6,\ 7,\ 10,\ 14,\ 15,\ \text{and }30,
\end{equation}
whose least common multiple is $210$. This gives a single universal modulus that detects whether a congruence obstruction exists.

\begin{corollary}\label{C:universal210}
Let \(E/\Q\) be a non-CM elliptic curve. Then \(E\) has a congruence obstruction modulo some integer \(m\geq2\) if and only if it has a congruence obstruction modulo \(210\).
\end{corollary}

This corollary has a consequence for the constant appearing in the quantitative version of the Koblitz--Zywina conjecture, which predicts that
\begin{equation}\label{E:KZ-Quant}
\pi_E^{\Prime}(x)
\sim
C_E\frac{x}{(\log x)^2}
\qquad\text{as }x\to\infty,
\end{equation}
where we write $C_E \coloneqq C_{E,1}$, with $C_{E,t}$ as in \cite[Definition~2.1]{MR2805578}. When $C_E = 0$, the asymptotic \eqref{E:KZ-Quant} is interpreted as indicating that $\pi_E^{\Prime}(x)$ is bounded as $x \to \infty$. Zywina shows that $C_E=0$ precisely when $E$ has a congruence obstruction modulo some integer $m \geq 2$ \cite[Section~2.1]{MR2805578}. Corollary~\ref{C:universal210} therefore implies that, for every non-CM elliptic curve $E/\Q$,
\[
C_E=0
\quad\Longleftrightarrow\quad
E\text{ has a congruence obstruction modulo }210.
\]
Equivalently, the positivity of $C_E$  is determined entirely by the mod $210$ Galois image. Thus, Corollary~\ref{C:universal210}  gives a finite level answer to Jones's question concerning criteria for the positivity of the Koblitz--Zywina constant \cite[Question~A12]{MR2727656}. Since the mod $210$ Galois image can be computed using Zywina's code \cite{ZywinaOpen, ZywinaOpenImage}, the resulting criterion is effective.

One may also express the criterion directly in terms of a single reduction: For non-CM elliptic curves $E/\Q$, we have \(C_E>0\) if and only if there exists a prime \(p>7\) of good reduction such that
\[
\gcd(\#E_p(\F_p),210)=1.
\]
Indeed, such a prime gives a Frobenius element in the mod \(210\) Galois image lying outside \(\mathcal I_1(210)\), where $\mathcal{I}_1$ is as in \eqref{E:I1}. Conversely, if the mod \(210\) image contains such an element, the Chebotarev density theorem implies the existence of infinitely many primes with the same property. An effective version of the Chebotarev density theorem, for example \cite{MR1355006} under the generalized Riemann hypothesis (GRH), can be used to bound the least such prime \(p\) in terms of \(N_E\), although we do not pursue this direction here.

Related questions concerning the vanishing of constants arising in other elliptic curve conjectures have also been studied. For the cyclicity conjecture, which concerns the distribution of primes \(p\) of good reduction for which \(E_p(\F_p)\) is cyclic, Serre observed that the corresponding constant vanishes precisely when \(\Q(E[2])=\Q\), or equivalently, when \(E\) has full rational \(2\)-torsion \cite[pp.~465--466]{MR3223094}. More recently, Jones and Vissuet \cite{MR4928015} studied the vanishing of Lang--Trotter constants and classified the maximal missing trace groups whose associated modular curves have genus zero, including composite level obstructions arising from entanglements among division fields.

The factor $x/(\log x)^2$ in \eqref{E:KZ-Quant} arises from the same heuristic that underlies the twin prime conjecture: one factor of $1/\log x$ comes from requiring $p$ to be prime, and a second comes from requiring the integer $\#E_p(\F_p)$, whose size is comparable to $p$ by the Hasse bound, to be prime. For this reason, the Koblitz--Zywina conjecture is often viewed as an elliptic curve analogue of the twin prime conjecture; see, for example, \cite{MR2843097,MR2879973}.

Even the qualitative form of the Koblitz--Zywina conjecture remains open, but several lines of inquiry provide supporting evidence for it. One approach considers averages over families. Balog, Cojocaru, and David \cite{MR2843097} proved that the predicted asymptotic holds on average over a two-parameter family of elliptic curves, and Jones \cite{MR2534114} studied the associated constants on average. Another direction replaces ``prime'' by ``almost prime'' in the spirit of Chen's theorem \cite{MR207668, MR434997} related to the twin prime conjecture. Assuming GRH, Miri and Murty \cite{MR1934487} obtained an almost prime result for non-CM elliptic curves, and  Steuding and Weng \cite{MR2140162,MR2189069} sharpened this to at most $9$ prime factors, counted with multiplicity. David and Wu \cite{MR2879973} subsequently proved, again under GRH, that if $E$ is non-CM  and has no congruence obstruction, then there are infinitely many primes $p$ for which $\#E_p(\F_p)$ has at most $8$ prime factors, counted with multiplicity.

A related question concerns the least prime divisor of $\#E_p(\F_p)$. For a prime $p\geq5$ of good reduction, the Hasse bound implies that $\#E_p(\F_p)>1$, so we may define
\[
c_E(p)
\coloneqq
\text{the least prime divisor of }\#E_p(\F_p).
\]
Jones proved that $c_E(p)$ is unbounded as $p$ varies over the primes $p \geq 5$ of good reduction, provided that $E$ has no congruence obstruction \cite[Appendix]{MR2727656}. Our second main theorem, which applies to both CM and non-CM elliptic curves, gives a quantitative version of this result.

\begin{theorem}\label{Quanticep}
Assume GRH. Let $E/\Q$ be an elliptic curve with no congruence obstruction modulo any integer $m\geq2$. For every real number $0<\kappa<1/8$, there are infinitely many primes $p\geq5$ of good reduction for $E$ such that
\[
c_E(p)>\kappa\log p.
\]
\end{theorem}

In Theorem~\ref{Quanticep}, by GRH we mean the generalized Riemann hypothesis for the Dedekind zeta functions of the number fields to which \cite{MR1355006} is applied through Theorem~\ref{ECDTA}. Theorem~\ref{Quanticep} may be viewed as an elliptic curve analogue of classical results on the least prime divisor of $p+2$ for infinitely many primes $p$; see \cite{MR12111,MR12112} and the subsequent sieve-theoretic improvements discussed in \cite[Section~6]{MR2245880}. It is conceivable that sieve-theoretic methods could yield a stronger bound than Theorem~\ref{Quanticep}, but we do not pursue this direction here.

We conclude the introduction by describing the organization of the paper. In Section~\ref{S:Prelims}, we recall the necessary background on Galois representations, congruence obstructions, fiber products, modular curves, and prime level images. In Section~\ref{S:BoundLargest}, we bound the prime divisors of a primitive obstruction level, proving that the largest prime divisor is at most $13$ and is not equal to $11$. In Section~\ref{S:prime-restrictions}, we obtain further restrictions on the images modulo $2$, $3$, $5$, $7$, and $13$, reducing the levels and images that must be considered. In Section~\ref{S:Search}, we carry out a finite group-theoretic search, which reduces the classification to a finite collection of modular curves. In Section~\ref{S:RatPtWork}, we determine the rational points on these modular curves, analyze the relevant quadratic twists, and complete the proof of Theorem~\ref{T:class}. In Section~\ref{S:set}, we study the primitive obstruction set and prove the universal level $210$ criterion. Section~\ref{S:NumEx} gives an explicit example, where we explain a primitive level $30$ obstruction in terms of division field entanglements. Finally, in Section~\ref{S:ProofOfQuant}, we prove Theorem~\ref{Quanticep} by applying an effective Chebotarev theorem with avoidance.

The computations in this article were carried out using \texttt{Magma} V2.28-21 \cite{MR1484478}. The accompanying code is available in the online repository: \\
\centerline{\url{https://github.com/Rakvi6893/Vanishing-of-Koblitz-s-constant}.}

\section{Preliminaries} \label{S:Prelims}

\subsection{Galois representations of elliptic curves}

Let \(E/\Q\) be an elliptic curve. For each positive integer \(n\), let \(E[n]\) denote the subgroup of \(E(\overline{\Q})\) consisting of points whose order divides \(n\), and let \(\Q(E[n])\) denote the \(n\)-division field of \(E\), which is a finite Galois extension of $\Q$. The natural action of \(\Gal(\overline{\Q}/\Q)\) on \(E[n]\) gives rise to the \emph{mod \(n\) Galois representation}
\[
\rho_{E,n}\colon \Gal(\overline{\Q}/\Q)\longrightarrow \Aut(E[n]).
\]
By choosing a basis for the free \(\Z/n\Z\)-module \(E[n]\), we identify \(\Aut(E[n])\) with \(\GL_2(\Z/n\Z)\). We write \(G_E(n)\) for the image of \(\rho_{E,n}\) in $\GL_2(\Z/n\Z)$, which is well defined up to conjugacy. The Weil pairing identifies $\det \rho_{E,n}$ with the mod $n$ cyclotomic character. In particular, $\det G_E(n) = (\Z/n\Z)^\times$.

Choosing compatible bases as $n$ varies gives the \emph{adelic Galois representation}
\[
\rho_E\colon\Gal(\overline{\Q}/\Q)\longrightarrow
\GL_2(\widehat{\Z}),
\]
whose image we denote by $G_E$; here $\widehat{\Z}$ denotes the ring of profinite integers. If $E$ is non-CM, Serre's open image theorem \cite{MR387283} gives that $G_E$ is open in $\GL_2(\widehat{\Z})$. For such $E$, the least positive integer $N$ such that $G_E$ is the full inverse image of $G_E(N)$ under reduction modulo $N$ is called the \emph{adelic level} of $E$.

Let \(p\) be a prime of good reduction for \(E\). As in the introduction, we write \(E_p\) for the reduction of \(E\) modulo \(p\), and
\[
\#E_p(\F_p)=p+1-a_p(E),
\]
where \(a_p(E)\in \Z\) is the trace of Frobenius. If \(m\geq 1\) and \(p\nmid mN_E\), then
\[
\det \rho_{E,m}(\Frob_p)\equiv p \pmod m \quad\text{and}\quad \tr \rho_{E,m}(\Frob_p)\equiv a_p(E) \pmod m,
\]
where \(\Frob_p \in \Gal(\overline{\Q}/\Q)\) denotes a choice of arithmetic Frobenius at \(p\). Since 
\[ \det(I-M)=1-\tr(M)+\det(M)\] 
for every \(2\times 2\) matrix \(M\), it follows that
\[
\#E_p(\F_p)\equiv \det\bigl(I-\rho_{E,m}(\Frob_p)\bigr)\pmod m.
\]
Alternatively, this congruence can be derived directly using the Frobenius endomorphism. Let $\pi_p\colon E_p \to E_p$ be the Frobenius map $(x,y) \mapsto (x^p,y^p)$. Then 
$$E_p(\F_p) = \ker(1-\pi_p).$$
Since $p \nmid mN_E$, reduction induces an isomorphism $E[m] \simeq E_p[m]$ under which the restriction of $\pi_p$ to $E_p[m]$ is identified with the action of $\rho_{E,m}(\Frob_p)$ on $E[m]$; see \cite[p.~IV-5]{MR1484415}. Therefore,
$$\#E_p(\F_p) = \deg(1-\pi_p) \equiv \det(I-\rho_{E,m}(\Frob_p)) \pmod m .$$

\subsection{Congruence obstructions via Galois representations}

For $m \geq 2$, define
\begin{equation} \label{E:I1}
\mathcal I_1(m)\coloneqq \{M\in \GL_2(\Z/m\Z): \det(I-M)\notin (\Z/m\Z)^\times\} \subseteq  \GL_2(\Z/m\Z),
\end{equation}
which is invariant under conjugation. For every prime \(p\nmid mN_E\),  the preceding discussion gives that
\begin{equation}\label{E:equiv}
\gcd(\#E_p(\F_p),m)>1 \quad\Longleftrightarrow\quad \rho_{E,m}(\Frob_p)\in \mathcal I_1(m).
\end{equation}
When \(m=\ell\) is prime, a matrix in $\GL_2(\F_\ell)$ belongs to \(\mathcal I_1(\ell)\) precisely when \(1\) is an eigenvalue.

This gives the following Galois-theoretic formulation of congruence obstructions.
\begin{definition}\label{Def:GpObstruction}
Let \(m\geq 2\). We say that \(E\) has a \emph{congruence obstruction modulo} \(m\) if
\[
G_E(m)\subseteq \mathcal I_1(m).
\]
If, moreover, \(E\) has no congruence obstruction modulo any proper divisor \(n\geq 2\) of \(m\), then we say that the congruence obstruction modulo \(m\) is \emph{primitive}.
\end{definition}

By \eqref{E:equiv} and the Chebotarev density theorem, Definition~\ref{Def:GpObstruction} is equivalent to the definition given in the introduction (Definition~\ref{D:CongObs}). Indeed, if
\(
G_E(m)\subseteq\mathcal I_1(m),
\)
then \(\gcd(\#E_p(\F_p),m)>1\) for every prime \(p\nmid mN_E\). Conversely, suppose that some \(g\in G_E(m)\) lies outside \(\mathcal I_1(m)\). Since \(\mathcal I_1(m)\) is invariant under conjugation, the conjugacy class of \(g\) is disjoint from \(\mathcal I_1(m)\). By the Chebotarev density theorem, this conjugacy class occurs as the Frobenius conjugacy class for a positive density set of primes \(p \nmid m N_E\), and \eqref{E:equiv} gives that
\(
\gcd(\#E_p(\F_p),m)=1
\)
for all such primes \(p\).

We also use the same terminology for arbitrary subgroups of \(\GL_2(\Z/m\Z)\). If \(G\subseteq\GL_2(\Z/m\Z)\) and \(n\mid m\), let \(G(n)\) denote the image of \(G\) under reduction modulo \(n\). We say that \(G\) has a \emph{congruence obstruction modulo \(m\)} if
\(
G\subseteq\mathcal I_1(m),
\)
and that the obstruction is \emph{primitive} if 
\(
G(n)\not\subseteq\mathcal I_1(n)
\)
for every proper divisor \(n\geq 2\) of \(m\).

Finally, for a subgroup \(G\subseteq \GL_2(\Z/m\Z)\), we define
\begin{equation}\label{eq:F1}
\mathcal F_1(G)\coloneqq \frac{|G\cap \mathcal I_1(m)|}{|G|}.
\end{equation}
The modulus in \(\mathcal F_1\) will always be clear from the ambient group. Observe that \(G\) has a congruence obstruction modulo \(m\) if and only if \(\mathcal F_1(G)=1\). Moreover, the obstruction is primitive if and only if \(\mathcal F_1(G)=1\) and \(\mathcal F_1(G(n))<1\) for every maximal proper divisor \(n\geq 2\) of \(m\). 

\subsection{Subgroups at prime level} For every prime $\ell$, let $B_0(\ell)$ denote the Borel subgroup of upper triangular matrices in $\GL_2(\F_\ell)$. In the Cartan definitions that follow, suppose that $\ell$ is odd. The split Cartan subgroup is the subgroup consisting of diagonal matrices, denoted $C_{\rm s}(\ell)$, and its normalizer is denoted $C_{\rm s}^+(\ell)$. Fix a nonsquare $\epsilon\in\F_\ell^\times$. The nonsplit Cartan subgroup is
\[
C_{\rm ns}(\ell)
\coloneqq
\left\{
\begin{pmatrix}a&\epsilon b\\ b&a\end{pmatrix}
: a,b \in \mathbb{F}_\ell, \, (a,b)\neq(0,0)
\right\}
\]
and its normalizer is
\[
C_{\rm ns}^+(\ell)
=C_{\rm ns}(\ell)\sqcup\omega C_{\rm ns}(\ell),
\qquad
\omega=\begin{pmatrix}1&0\\0&-1\end{pmatrix}.
\]

Dickson's classification implies that, up to conjugacy, every proper subgroup of \(\GL_2(\F_\ell)\) is contained in a Borel subgroup or in the normalizer of a Cartan subgroup, has projective image isomorphic to \(A_4\), \(S_4\), or \(A_5\), or contains \(\SL_2(\F_\ell)\); see \cite[Section~2]{MR387283} or \cite[Section~3]{MR3482279}. We will sometimes use Sutherland's labels for subgroups of $\GL_2(\F_\ell)$ from \cite{MR3482279}, as well as the LMFDB labels for subgroups of \(\GL_2(\widehat{\Z})\) \cite{lmfdb}.

We now record a useful observation relating to the nonsplit Cartan subgroup.
\begin{lemma}\label{L:I1Cns}
Let $\ell$ be an odd prime.
\begin{enumerate}
\item $\mathcal I_1(\ell)\cap C_{\rm ns}(\ell)=\{I\}$.
\item If $A\in\mathcal I_1(\ell)\cap(C_{\rm ns}^+(\ell)\setminus C_{\rm ns}(\ell))$, then $\det A=-1$.
\end{enumerate}
\end{lemma}

\begin{proof}
If a matrix
\(
\begin{psmallmatrix}a&\epsilon b\\ b&a\end{psmallmatrix}
\in C_{\rm ns}(\ell)
\)
is contained in $\mathcal I_1(\ell)$, then
\[(1-a)^2-\epsilon b^2=0.\]
If $b\neq0$, then
\(
\epsilon=((1-a)/b)^2,
\)
contradicting the fact that $\epsilon$ is a nonsquare. Hence $b=0$ and the displayed equation then gives $a=1$. This proves the first assertion.

Every element of $C_{\rm ns}^+(\ell)\setminus C_{\rm ns}(\ell)$ has the form
\(
\begin{psmallmatrix}a&\epsilon b\\-b&-a\end{psmallmatrix}
\)
and therefore has trace zero. If such an element $A$ lies in $\mathcal I_1(\ell)$, then
\[
0=\det(I-A)=1-\tr(A)+\det(A)=1+\det(A),
\]
so $\det(A)=-1$, proving the second assertion.
\end{proof}

\subsection{Fiber products}

Let $G_1$ and $G_2$ be finite groups equipped with surjective homomorphisms $\phi_i\colon G_i\twoheadrightarrow Q$ onto a common group $Q$. The \emph{fiber product} of $G_1$ and $G_2$ with respect to $(\phi_1, \phi_2)$ is the subgroup of $G_1 \times G_2$ defined by
\[
G_1\times_{(\phi_1,\phi_2)}G_2
\coloneqq
\{(g_1,g_2)\in G_1\times G_2:\phi_1(g_1)=\phi_2(g_2)\}.
\]
Goursat's lemma states that every subgroup $H\subseteq G_1\times G_2$ that projects surjectively onto both factors can be realized as a fiber product of $G_1$ and $G_2$ for suitable $Q, \phi_1,$ and $\phi_2$; see, for example, \cite[p.~75]{MR1878556}. 

We will also use the corresponding field-theoretic description. If \(L_1/K\) and \(L_2/K\) are finite Galois extensions in a common algebraic closure of $K$, then restriction gives a natural isomorphism
\[
\Gal(L_1L_2/K)
\simeq
\Gal(L_1/K)\times_{\Gal(L_1\cap L_2/K)}\Gal(L_2/K),
\]
where the two maps to \(\Gal(L_1\cap L_2/K)\) are given by restriction.

\subsection{Modular curves and prime level images} \label{S:ModCrvs}

Let \(G\subseteq\GL_2(\widehat{\Z})\) be an open subgroup with full determinant, that is, with
\(
\det(G)=\widehat{\Z}^{\times},
\)
and let \(N\) be the level of \(G\). When convenient, we identify \(G\) with its reduction modulo \(N\), viewed as a subgroup of \(\GL_2(\Z/N\Z)\). Associated to \(G\) is a smooth, projective, geometrically integral modular curve \(X_G\) over \(\Q\). The inclusion
\(
G\subseteq\GL_2(\Z/N\Z)
\)
induces a natural morphism
\[
j\colon X_G\longrightarrow X(1)\isom\mathbb P^1,
\]
called the \emph{\(j\)-map}. A \emph{cusp} of \(X_G\) is a point whose image under \(j\) is \((1:0)\).

We recall the interpretation of the non-cuspidal rational points of \(X_G\) in terms of Galois representations. Suppose first that \(-I\in G\). If \(E/\Q\) is an elliptic curve with
\(
j(E)\notin\{0,1728\},
\)
then \(G_E(N)\) is conjugate to a subgroup of \(G\) if and only if
\(
j(E)\in j(X_G(\Q)).
\)
Now suppose that \(-I\notin G\), and set $G^{\pm} = \langle G, -I \rangle$. Then $X_G$ agrees with the coarse modular curve $X_{G^{\pm}}$, compatibly with the $j$-map. In this case, containment of the mod $N$ image in \(G\) is no longer determined by the \(j\)-invariant alone: for an elliptic curve \(E/\Q\) with \(j(E)\notin\{0,1728\}\), one has
\(
j(E)\in j(X_G(\Q))
\)
if and only if some quadratic twist \(E^{(d)}/\Q\) of \(E\) has mod \(N\) image conjugate to a subgroup of \(G\); see \cite[Lemma 5.1]{MR3500996}.

When we need to compute a model of a modular curve, we use Zywina's code from \cite{ZywinaOpen}. This code uses the convention that for $-I \in G$ and $j(E) \not\in \{0, 1728\}$, the modular curve associated with  \(G\) has a rational point above $j(E)$ if and only if $G_E(N)$ is conjugate to a subgroup of the transpose
\(
G^{\intercal}=\{g^{\intercal}:g\in G\}.
\)
Consequently, to compute the modular curve associated with \(G\) in our convention, we pass \(G^{\intercal}\) to Zywina's code; see \cite[Remark~3.6]{ZywinaOpen}.

We will use the following theorem, which summarizes the known unconditional classification of mod \(\ell\) images. It combines numerous results that determine rational points on the relevant modular curves.

\begin{theorem}\label{T:SOTA}
Let \(E/\Q\) be a non-CM elliptic curve, and let \(\ell\) be a prime. Suppose that
\(
G_E(\ell)\neq\GL_2(\F_\ell).
\)
Then either \(G_E(\ell)\) is conjugate to one of the \(63\) proper subgroups listed in \cite[Table~3]{MR3482279}, or \(\ell\geq 19\) and \(G_E(\ell)\) is conjugate to \(C_{\rm ns}^+(\ell)\).
\end{theorem}
\begin{proof}
For $\ell\leq11$, the result follows from the classification recorded in \cite[Section~1 and Table~3]{MR3482279}. For $\ell=13, 17$, the remaining cases are settled by the determination of the rational points on the relevant modular curves in \cite{MR3961086, MR4589060}. Thus every proper mod $\ell$ image for $\ell\leq17$ is conjugate to one of the groups listed in \cite[Table~3]{MR3482279}.

Now suppose that $\ell\geq19$. If $\ell>37$, the result follows directly from \cite[Theorem~1.5]{furio2025serresuniformityquestionproper}, which shows that $G_E(\ell)$ is either $\GL_2(\F_\ell)$ or conjugate to $C_{\rm ns}^+(\ell)$. It remains to consider
\[
\ell\in\{19,23,29,31,37\}.
\]
By \cite[Proposition~1.13]{zywina2015possibleimagesmodell}, apart from the Borel cases at level $37$, the group $G_E(\ell)$ is conjugate either to $C_{\rm ns}^+(\ell)$ or to the index $3$ subgroup defined therein. The Borel images are among the groups listed in \cite[Table~3]{MR3482279}, and \cite[Theorem~1.6]{furio2025serresuniformityquestionproper} shows that the index $3$ subgroup does not occur for any prime $\ell>5$. The result follows.
\end{proof}

In fact, it is conjectured that the \(63\) groups appearing in \cite[Table~3]{MR3482279} give the complete list of non-surjective mod \(\ell\) images as \(\ell\) varies; see \cite[Conjecture~1.1]{MR3482279}.

\subsection{Prime level congruence obstructions}

\begin{proposition}\label{P:prime-level}
Let \(E/\Q\) be a non-CM elliptic curve, and let \(\ell\) be a prime. Then \(E\) has a congruence obstruction modulo \(\ell\) if and only if
\[
\ell\in\{2,3,5,7\}
\]
and, up to conjugacy, the mod $\ell$ Galois image of $E$ is one of the groups listed in Table \ref{tab:prime-groups}.
\end{proposition}

\begin{proof}
By  \cite[Theorem~2]{MR604840}, $E$ has a congruence obstruction modulo $\ell$ if and only if there exists an elliptic curve $E'/\Q$ that is $\Q$-isogenous to $E$ and has a rational point of order $\ell$. By Mazur's theorem on rational torsion \cite{MR488287}, such an $E'$ can exist only for $\ell \in\{2,3,5,7\}$. For such primes $\ell$, we consider all mod $\ell$ images that occur over $\Q$ (see Theorem~\ref{T:SOTA} in Section~\ref{S:ModCrvs}). We check which of these images satisfy $H \subseteq \mathcal{I}_1(\ell)$, and give an example realizing each such image. The associated code appears in the file \texttt{primecongobstruction.m}.
\end{proof}

The table below (Table~\ref{tab:prime-groups}) gives all mod $\ell$ Galois images for non-CM elliptic curves over $\Q$ with a congruence obstruction. The group is identified using its LMFDB label, as is the example elliptic curve provided.

\begin{table}[H]
\centering
\caption{Galois images mod $\ell$ with a congruence obstruction}
\label{tab:prime-groups}
\begin{tabular}{|c|c|c|}
\hline
$\ell$  & Group & Example \\ \hline
2  & \mc{2.3.0.a.1} & \ec{14.a1} \\
2  & \mc{2.6.0.a.1}    & \ec{15.a2}  \\ 
3  & \mc{3.8.0-3.a.1.2}    & \ec{14.a4} \\ 
3 &  \mc{3.8.0-3.a.1.1} & \ec{14.a2} \\ 
3 &  \mc{3.24.0-3.a.1.1} & \ec{14.a6} \\ 
5 &  \mc{5.24.0-5.a.1.2} & \ec{38.b2} \\ 
5 &  \mc{5.24.0-5.a.2.2} & \ec{38.b1}  \\ 
5 & \mc{5.120.0-5.a.1.2} & \ec{550.j3}  \\ 
7 & \mc{7.48.0-7.a.1.2} & \ec{26.b2} \\ 
7 & \mc{7.48.0-7.a.2.2} & \ec{26.b1} \\ \hline
\end{tabular}
\end{table}

\section{Bounding the prime divisors of primitive congruence obstruction levels} \label{S:BoundLargest}

In this section, we bound the prime divisors of a primitive congruence obstruction level. We first show, using fiber product arguments, that if its largest prime divisor $\ell$ is greater than $37$, then the mod $\ell$ image  can be neither the full group nor the normalizer of a nonsplit Cartan subgroup,
contradicting Theorem~\ref{T:SOTA}. This gives a uniform bound of \(37\). We then combine this bound with known rational point results on several modular curves to prove that the largest prime divisor is at most \(13\) and cannot equal \(11\).

We begin with the full image case.

\begin{lemma} \label{L:NoFiber}
Let $\ell \geq 5$ be a prime, and let $\phi \colon \GL_2(\F_\ell) \twoheadrightarrow Q$ be a surjective homomorphism. If some fiber of $\phi$ is contained in $\mathcal{I}_1(\ell)$, then $\ell$ divides $|Q|$.
\end{lemma}

\begin{proof}
We prove the contrapositive. Suppose that \(\ell\) does not divide \(|Q|\). For $\ell\geq5$, the group $\PSL_2(\F_\ell)$ is simple and $\SL_2(\F_\ell)$ is perfect; see  \cite[Chapter~XIII, Theorems~8.3 and~8.4]{MR1878556}. It follows that every proper normal subgroup of $\SL_2(\F_\ell)$ is contained in its center. In particular, every nontrivial quotient of $\SL_2(\F_\ell)$ has order divisible by $\ell$.

Now \(\phi(\SL_2(\F_\ell))\) is isomorphic to both a quotient of \(\SL_2(\F_\ell)\) and a subgroup of \(Q\). Since $\ell \nmid |Q|$, it follows that \(\phi(\SL_2(\F_\ell))\) is trivial. Hence
\[
        \SL_2(\F_\ell)\subseteq\ker\phi.
\]
Thus \(\phi\) factors through the determinant map \(\det\colon \GL_2(\F_\ell)\to \F_\ell^\times\). As such, there exists a surjective homomorphism \(\upsilon\colon\F_\ell^\times\twoheadrightarrow Q\) such that \(\phi=\upsilon\circ\det\).

Now take an arbitrary fiber \(\phi^{-1}(\gamma)\), where \(\gamma\in Q\). Since \(\upsilon\) is surjective, there exists \(a\in\F_\ell^\times\) such that \(\upsilon(a)=\gamma\). For the matrix
\(
        M_a\coloneqq
        \begin{psmallmatrix}
                0 & -a\\
                1 & 1
        \end{psmallmatrix},
\)
we have \(\det(M_a)=a\). Thus \(M_a\in\GL_2(\F_\ell)\), and
\[
        \phi(M_a)=\upsilon(\det M_a)=\upsilon(a)=\gamma.
\]
On the other hand, \(\det(I-M_a)=a\), so \(M_a\notin\mathcal I_1(\ell)\). Thus every fiber of \(\phi\) contains an element outside \(\mathcal I_1(\ell)\).
\end{proof}

\begin{lemma} \label{L:remove-full-prime}
Let $E/\Q$ be an elliptic curve that has a primitive congruence obstruction modulo a squarefree composite integer $m \geq 2$. Let $\ell\ge 5$ be a prime divisor of $m$, and write $m = n \ell$. If $G_E(\ell)=\GL_2(\F_\ell)$, then $\ell$ divides $|G_E(n)|$.
\end{lemma}

\begin{proof}
Seeking a contradiction, suppose that $\ell\nmid |G_E(n)|$.  Since the obstruction modulo $m$ is primitive, $E$ does not have a congruence obstruction modulo $n$. Hence there exists a matrix $A_n\in G_E(n)\setminus\mathcal I_1(n)$. By Goursat's lemma, we may write
\[
    G_E(m)=\GL_2(\F_\ell)\times_{(\phi,\psi)}G_E(n)
\]
for some surjective homomorphisms $\phi \colon \GL_2(\F_\ell) \to Q$ and $\psi \colon G_E(n) \to Q$, where $Q$ is some finite group. Since $Q$ is isomorphic to a quotient of $G_E(n)$, our assumption that $\ell\nmid |G_E(n)|$ implies $\ell\nmid |Q|$.

Since $A_n\in G_E(n)\setminus\mathcal I_1(n)$, the fiber $\phi^{-1}(\psi(A_n))$ must be contained in $\mathcal I_1(\ell)$. Indeed, if $A_\ell\in\phi^{-1}(\psi(A_n))$, then $(A_\ell,A_n)\in G_E(m)\subseteq\mathcal I_1(m)$, and the Chinese remainder theorem implies that $A_\ell\in\mathcal I_1(\ell)$. This contradicts Lemma~\ref{L:NoFiber}.
\end{proof}

\begin{corollary} \label{C:full-prime-support}
Let \(E/\Q\) be an elliptic curve with a primitive congruence obstruction modulo a squarefree composite integer \(m\), and let \(\ell\geq 5\) be a prime divisor of \(m\). If \(G_E(\ell)=\GL_2(\F_\ell)\), then there is a prime \(p>\ell\) dividing \(m\) such that \(\ell\) divides \(|G_E(p)|\). In particular, if \(\ell\) is the largest prime divisor of \(m\), then \(G_E(\ell)\neq\GL_2(\F_\ell)\).
\end{corollary}

\begin{proof}
Write \(m=n\ell\). By Lemma~\ref{L:remove-full-prime}, we have that \(\ell\) divides \(|G_E(n)|\). Since \(n\) is squarefree, under the isomorphism of the Chinese remainder theorem,
\[
G_E(n)\subseteq\prod_{p\mid n}G_E(p).
\]
It follows that \(\ell\) divides \(|G_E(p)|\) for some prime \(p\) dividing \(n\).

We claim that \(p>\ell\). Indeed, if \(p<\ell\), then $|G_E(p)|$ divides
\[
|\GL_2(\F_p)|=p(p-1)^2(p+1).
\]
Since \(\ell\geq 5\), we have \(p+1<\ell\), as equality would imply that \(p=\ell-1\) is even and greater than \(2\). Hence \(\ell\) divides none of \(p\), \(p-1\), or \(p+1\), contradicting that \(\ell\) divides \(|G_E(p)|\).
\end{proof}

We next treat the case of the normalizer of the nonsplit Cartan.

\begin{lemma} \label{L:Nn-fiber-det}
Let \(\ell\geq 5\) be a prime, and let
\(
        \phi \colon C_{\rm ns}^+(\ell)\twoheadrightarrow Q
\)
be a surjective homomorphism. If some fiber of \(\phi\) is contained in
\(\mathcal I_1(\ell)\), then
\(
        \ker\phi\subseteq \SL_2(\F_\ell).
\)
\end{lemma}
\begin{proof}
We first note that no nontrivial normal subgroup of \(C_{\rm ns}^+(\ell)\)
is contained in \(\mathcal I_1(\ell)\).  Indeed, suppose that
\(M\) is a normal subgroup of \(C_{\rm ns}^+(\ell)\) and that \(M\subseteq\mathcal I_1(\ell)\). Then, by Lemma~\ref{L:I1Cns},
\[
        M\cap C_{\rm ns}(\ell)=\{I\}.
\]
If \(x\in M\setminus\{I\}\), then \(x\in C_{\rm ns}^+(\ell)\setminus C_{\rm ns}(\ell)\).
For every \(c\in C_{\rm ns}(\ell)\), we have
\[
        cxc^{-1}x^{-1} \in M\cap C_{\rm ns}(\ell)=\{I\},
\]
since both \(M\) and \(C_{\rm ns}(\ell)\) are normal in \(C_{\rm ns}^+(\ell)\). Thus \(x\) centralizes
\(C_{\rm ns}(\ell)\), but the
centralizer of \(C_{\rm ns}(\ell)\) in \(\GL_2(\F_\ell)\) is \(C_{\rm ns}(\ell)\) as noted in \cite[\S~2.1]{Zywina2011}. This is a contradiction, so \(M=\{I\}\).

Write $N \coloneqq \ker \phi$, and let \(gN\) be a fiber of \(\phi\) contained in \(\mathcal I_1(\ell)\).
If \(I\in gN\), then \(gN=N\), so \(N\subseteq\mathcal I_1(\ell)\), and the
preceding paragraph gives \(N=\{I\} \subseteq \SL_2(\F_\ell)\).

We may therefore assume that \(I\notin gN\).  Since  \(gN \subseteq C_{\rm ns}^+(\ell)\cap\mathcal I_1(\ell)\), every element of $gN$ lies in \(C_{\rm ns}^+(\ell) \setminus C_{\rm ns}(\ell)\) by Lemma~\ref{L:I1Cns}. Thus we have \(\det g=-1\) and \(\det(gn)=-1\) for every \(n\in N\). It follows that \(\det n=1\) for all $n \in N$, so \(N\subseteq\SL_2(\F_\ell)\).
\end{proof}

We now give a variant of \cite[Proposition 1.6]{MR4588927}.

\begin{lemma} \label{L:NoCyc}
Let \(E/\Q\) be an elliptic curve, let \(\ell\geq 11\) be a prime, and let \(n\geq 1\) be an integer such that \(\ell \nmid n\). If \(E\) has potentially good reduction at \(\ell\), then
\[
        \Q(\zeta_\ell)\not\subseteq \Q(E[n]),
\]
where $\zeta_\ell$ is a primitive $\ell$-th root of unity.
\end{lemma}
\begin{proof}
By \cite[Section~5.6]{MR387283}, after a finite base change \(K/\Q\), the curve \(E\) attains good reduction at a prime \(\lambda\) of \(K\) above \(\ell\), with ramification index \(e(\lambda/\ell)\leq 6\). Since \(\ell\nmid n\), the N\'eron--Ogg--Shafarevich criterion implies that \(K(E[n])/K\) is unramified above \(\lambda\). Choose a prime \(\mu\) of \(K(E[n])\) above \(\lambda\), and let \(\nu\) be its restriction to \(\Q(E[n])\). Then
\[
e(\nu/\ell) \leq e(\mu/\ell) = e(\mu/\lambda)e(\lambda/\ell) \leq 6.
\]
On the other hand, \(\ell\) is totally ramified in \(\Q(\zeta_\ell)/\Q\), with ramification index \(\ell-1\). Thus, if \(\Q(\zeta_\ell)\subseteq\Q(E[n])\), then the ramification index of every prime of \(\Q(E[n])\) above \(\ell\) is divisible by \(\ell-1\). This would give \(\ell-1\leq 6\), contradicting \(\ell\geq 11\).
\end{proof}

\begin{proposition} \label{P:exclude-nonsplit-prime}
Let \(E/\Q\) be an elliptic curve with a primitive congruence obstruction modulo a squarefree composite integer \(m\), and let \(\ell\geq 11\) be a prime divisor of \(m\). Then \(G_E(\ell)\) is not conjugate to \(C_{\rm ns}^+(\ell)\).
\end{proposition}

\begin{proof}
Write \(m=n\ell\), and suppose that \(G_E(\ell)\) is conjugate to \(C_{\rm ns}^+(\ell)\). After changing the basis of \(E[\ell]\), we may assume that \(G_E(\ell)=C_{\rm ns}^+(\ell)\). By Goursat's lemma, there are surjective homomorphisms \(\phi\colon G_E(\ell)\twoheadrightarrow Q\) and \(\psi\colon G_E(n)\twoheadrightarrow Q\) such that
\[
G_E(m)=G_E(\ell)\times_{(\phi,\psi)}G_E(n).
\]

Since the congruence obstruction modulo \(m\) is primitive, \(E\) does not have a congruence obstruction modulo \(n\). Hence we may choose \(A_n\in G_E(n)\setminus\mathcal I_1(n)\). Set \(\gamma=\psi(A_n)\). We claim that \(\phi^{-1}(\gamma)\subseteq\mathcal I_1(\ell)\). Indeed, if \(A_\ell\in\phi^{-1}(\gamma)\), then \((A_\ell,A_n)\in G_E(m)\subseteq\mathcal I_1(m)\). Since \(A_n\notin\mathcal I_1(n)\), the Chinese remainder theorem implies that \(A_\ell\in\mathcal I_1(\ell)\).

It follows from Lemma~\ref{L:Nn-fiber-det} that \(\ker\phi\subseteq\SL_2(\F_\ell)\). By the Weil pairing,
\[
\Q(E[\ell])^{G_E(\ell)\cap\SL_2(\F_\ell)}=\Q(\zeta_\ell).
\]
The field-theoretic interpretation of Goursat's lemma therefore gives
\[
\Q(\zeta_\ell)
\subseteq \Q(E[\ell])^{\ker\phi}
=\Q(E[\ell])\cap\Q(E[n])
\subseteq\Q(E[n]).
\]
Since \(G_E(\ell)\) is contained in the normalizer of a nonsplit Cartan subgroup, \(E\) has potentially good reduction at \(\ell\) by \cite[Proposition 2.2]{MR3885140}. This contradicts Lemma~\ref{L:NoCyc}.
\end{proof}

We can now obtain a uniform bound.

\begin{proposition} \label{P:sqfree}
Let \(E/\Q\) be a non-CM elliptic curve, and let \(m\geq 2\) be squarefree. If \(E\) has a primitive congruence obstruction modulo \(m\), then every prime divisor of \(m\) is at most \(37\).
\end{proposition}

\begin{proof}
If $m$ is prime, then $m \leq 7$ by Proposition~\ref{P:prime-level}. We may therefore assume that \(m\) is composite. Let \(\ell\) be the largest prime divisor of \(m\), and suppose, for a contradiction, that \(\ell>37\). By Theorem~\ref{T:SOTA}, either \(G_E(\ell)=\GL_2(\F_\ell)\), or \(G_E(\ell)\) is conjugate to \(C_{\rm ns}^+(\ell)\). The first case is ruled out by Corollary~\ref{C:full-prime-support}, and the second by Proposition~\ref{P:exclude-nonsplit-prime}. Therefore \(\ell\leq 37\).
\end{proof}

Proposition~\ref{P:sqfree} gives a uniform bound on the prime divisors of a primitive congruence obstruction level. We now sharpen this bound, which is necessary for the finite group-theoretic search. We begin with a criterion that allows us to rule out all quadratic twists of an elliptic curve at once.

\begin{proposition} \label{P:Twist}
Let \(E/\Q\) be a non-CM elliptic curve, and let \(N\) be a positive multiple of the
adelic level of \(E\). Suppose that there exists a matrix \(A\in G_E(N)\) such
that
\begin{equation} \label{E:lvlN}
            \det(I-A)\in(\Z/N\Z)^\times
        \quad\text{and}\quad
        \det(I+A)\in(\Z/N\Z)^\times .
\end{equation}
Then no quadratic twist of \(E\) has a congruence obstruction modulo any integer \(m\geq 2\).
\end{proposition} 
\begin{proof} Let  \(\chi\colon\Gal(\overline{\Q}/\Q)\to\{\pm1\}\) be a possibly trivial quadratic character, and let \(E^\chi\) be the corresponding quadratic twist. Since the existence of a congruence obstruction modulo an integer is equivalent to the existence of one modulo its radical, it suffices to consider squarefree moduli. Fix a squarefree integer $m \geq 2$. We will show that $E^{\chi}$ has no congruence obstruction modulo $m$. For the matrix
\(
        M \coloneqq
        \begin{psmallmatrix}
        0 & 1\\
        1 & 1
        \end{psmallmatrix},
\)
we have that \(\det(M)=-1\), \(\det(I-M)=-1\), and \(\det(I+M)=1\). Thus the reduction of \(M\) modulo \(p\) belongs to \(\GL_2(\F_p)\), and neither the reduction of \(M\) nor that of 
\(-M\) belongs to \(\mathcal{I}_1(p)\), for every prime \(p\).

Set \(n\coloneqq\lcm(N,m)\). By the Chinese remainder theorem, there is a matrix \(B\in\GL_2(\Z/n\Z)\) satisfying
\[
\begin{cases}
        B \equiv A\pmod N, \\
        B \equiv M\pmod p
        \qquad
        \text{for every prime \(p\mid m\) such that \(p\nmid N\)},
\end{cases}
\]
where the second family of congruences is understood to be empty when every prime divisor of \(m\) divides \(N\).

Because \(N\) is a multiple of the adelic level of \(E\), the group \(G_E(n)\) is the full inverse image of \(G_E(N)\) under reduction modulo \(N\). Hence \(B\in G_E(n)\), so there exists \(\sigma\in\Gal(\overline{\Q}/\Q)\) such that \(\rho_{E,n}(\sigma)=B\). After making compatible choices of bases for the $n$-torsion subgroups of \(E\) and \(E^\chi\), we have
\[
        \rho_{E^\chi,n}(\sigma)=\chi(\sigma)B.
\]
Let \(\overline B\) denote the reduction of \(B\) modulo \(m\). If \(p\mid m\) and \(p\mid N\), then \eqref{E:lvlN} gives that
\[
        \det(I-\overline B)\not\equiv 0\pmod p
        \quad \text{and} \quad
        \det(I+\overline B)\not\equiv 0\pmod p.
\]
If \(p\mid m\) and \(p\nmid N\), the same conclusion follows from the choice \(B\equiv M\pmod p\). Since \(m\) is squarefree, it follows that
\[
        \det(I-\overline B),\,
        \det(I+\overline B)
        \in(\Z/m\Z)^\times.
\]
The matrix \(\rho_{E^\chi,m}(\sigma)\) is either \(\overline B\) or \(-\overline B\), and in either case
\[
        \det\bigl(I-\rho_{E^\chi,m}(\sigma)\bigr)
        \in(\Z/m\Z)^\times.
\]
Thus \(G_{E^\chi}(m)\) contains an element outside \(\mathcal I_1(m)\), so \(E^\chi\) does not have a congruence obstruction modulo \(m\).
\end{proof}

\begin{lemma}\label{L:finite-j-exclusion}
Let \(j\in\Q\) be a non-CM \(j\)-invariant arising from a rational point on
one of the modular curves \(X_0(11)\), \(X_{S_4}(13)\), \(X_0(17)\), or \(X_0(37)\). Then no elliptic curve over \(\Q\) with \(j\)-invariant \(j\) has a congruence obstruction modulo any integer \(m \geq 2\).
\end{lemma}

\begin{proof}
The rational points on \(X_0(11)\), \(X_0(17)\), and \(X_0(37)\) are determined in \cite[Theorem~7.1]{MR482230}, and those on \(X_{S_4}(13)\) are determined in \cite[Section~5.1]{MR4589060}. For each resulting non-CM \(j\)-invariant, the computation in \texttt{RulingOutj.m} produces an elliptic curve \(E_j/\Q\) with $j(E_j) = j$, an integer \(N_j\) divisible by its adelic level, and \(A_j\in G_{E_j}(N_j)\) satisfying
\[
\det(I-A_j),\,\det(I+A_j)\in(\Z/N_j\Z)^\times.
\]
The result follows from Proposition~\ref{P:Twist}, since every elliptic curve over \(\Q\) with non-CM \(j\)-invariant \(j\) is a quadratic twist of \(E_j\).
\end{proof}

We now improve the bound obtained in Proposition~\ref{P:sqfree}.

\begin{proposition} \label{P:sqfree2}
Let \(E/\Q\) be a non-CM elliptic curve, and let \(m\geq 2\) be squarefree. If \(E\) has a primitive congruence obstruction modulo \(m\), then the largest prime divisor of $m$ is at most $13$ and is not equal to $11$.
\end{proposition}
\begin{proof}
Let \(\ell\) be the largest prime divisor of \(m\). If \(m\) is prime, then \(m\leq 7\), by Proposition~\ref{P:prime-level}. We may therefore assume that \(m\) is composite. By Proposition~\ref{P:sqfree}, we have
\(\ell\leq 37\).

Suppose first that \(\ell=37\). By Theorem~\ref{T:SOTA} and \cite[Table 3]{MR3482279}, the group \(G_E(37)\) is either \(\GL_2(\F_{37})\), conjugate to a subgroup of \(B_0(37)\), or conjugate to \(C_{\rm ns}^+(37)\). The first possibility is ruled out by Corollary~\ref{C:full-prime-support}, since \(37\) is the largest prime divisor of \(m\), and the third is ruled out by Proposition~\ref{P:exclude-nonsplit-prime}. Hence \(G_E(37)\) is conjugate to a subgroup of \(B_0(37)\). Thus \(j(E)\) arises from a non-CM rational point on \(X_0(37)\), contradicting Lemma~\ref{L:finite-j-exclusion}. Therefore \(37\) cannot be the largest prime divisor of \(m\).

Now suppose that \(\ell\in\{19,23,29,31\}.\) Theorem~\ref{T:SOTA} and \cite[Table 3]{MR3482279} show that \(G_E(\ell)\) is either \(\GL_2(\F_\ell)\) or conjugate to \(C_{\rm ns}^+(\ell)\). These possibilities are ruled out by Corollary~\ref{C:full-prime-support} and Proposition~\ref{P:exclude-nonsplit-prime}, respectively. Hence none of \(19,23,29,31\) can be the largest prime divisor of \(m\).

Next suppose that \(\ell=17\). Theorem~\ref{T:SOTA} and \cite[Table 3]{MR3482279} show that \(G_E(17)\) is either \(\GL_2(\F_{17})\) or conjugate to a subgroup of \(B_0(17)\). The former is excluded by Corollary~\ref{C:full-prime-support}, so \(j(E)\) must arise from a non-CM rational point on \(X_0(17)\), but this contradicts Lemma~\ref{L:finite-j-exclusion}.

Finally, suppose that \(\ell=11\). By Theorem~\ref{T:SOTA} and  \cite[Table 3]{MR3482279}, the only possibilities are the full image, a Borel image, and the normalizer of a nonsplit Cartan. The full image and normalizer of a nonsplit Cartan subgroup are excluded by Corollary~\ref{C:full-prime-support} and Proposition~\ref{P:exclude-nonsplit-prime}, while the Borel case is excluded by Lemma~\ref{L:finite-j-exclusion} applied to \(X_0(11)\).
\end{proof}

\section{Further restrictions on primitive obstruction levels and Galois images} \label{S:prime-restrictions}

We now prove the additional restrictions needed to make the finite group-theoretic search feasible. We first treat the cases in which the largest prime divisor is \(13\) or \(7\), then deduce a finite list of candidate composite levels, and finally obtain additional restrictions at level \(210\).

\begin{proposition} \label{P:13-large}
Let \(E/\Q\) be a non-CM elliptic curve with a primitive congruence obstruction modulo a squarefree integer \(m\). Suppose that \(13\) is the largest prime divisor of \(m\). Then \(m\) divides \(78\), \(G_E(2)=\GL_2(\F_2)\), \(G_E(3)=\GL_2(\F_3)\), and \(G_E(13)\) is conjugate to a subgroup of \(B_0(13)\).
\end{proposition}
\begin{proof}
Since a primitive obstruction at prime level has level at most \(7\) by Proposition~\ref{P:prime-level}, the integer \(m\) is composite. Theorem~\ref{T:SOTA} and \cite[Table~3]{MR3482279}, together with Corollary~\ref{C:full-prime-support}, leave two possibilities: \(G_E(13)\) is conjugate to a subgroup of \(B_0(13)\), or \(j(E)\) arises from a non-CM rational point on \(X_{S_4}(13)\). The second possibility is ruled out by Lemma~\ref{L:finite-j-exclusion}. After conjugating, we may therefore assume that
\(
G_E(13)\subseteq B_0(13).
\)
 Equivalently, \(E\) admits a rational \(13\)-isogeny. Thus, by \cite[Proposition~22(3)]{2512.00652}, the \(p\)-adic Galois representation of \(E\) is surjective for every prime \(p\notin\{2,13\}\). In particular,
\(
    G_E(3)=\GL_2(\F_3).
\)

To determine the mod \(2\) image, we first choose a single quadratic twist \(E'/\Q\) of \(E\) whose \(2\)-adic image contains \(-I\); see \cite[Corollary~5.25]{MR3482279}. The curve \(E'\) still admits a rational \(13\)-isogeny. Hence \cite[Proposition~22(1)]{2512.00652} implies that its \(2\)-adic image is conjugate either to \(\GL_2(\Z_2)\) or to the group with LMFDB label \mc{8.2.0.a.1}. Both of these groups reduce modulo \(2\) to \(\GL_2(\F_2)\). Quadratic twisting does not change the mod \(2\) image, since \(-I\equiv I\pmod 2\), and therefore
\(
    G_E(2)=\GL_2(\F_2).
\)

It remains to restrict the prime divisors of \(m\). By Proposition~\ref{P:sqfree2}, every prime divisor of \(m\) belongs to the set \(\{2,3,5,7,11,13\}\). Suppose that some \(q\in\{5,7,11\}\) divides \(m\), and choose the largest such \(q\). By \cite[Proposition~22(3)]{2512.00652},
\(
    G_E(q)=\GL_2(\F_q).
\)
Corollary~\ref{C:full-prime-support} therefore gives a prime \(p>q\) dividing \(m\) such that \(q\) divides the order of \(G_E(p)\). By the choice of \(q\), the only possible value of \(p\) is \(13\). On the other
hand, \(|G_E(13)|\) divides 
\[ |B_0(13)| =2^4\cdot 3^2\cdot 13, \]
which is divisible by none of \(5,7,11\), a contradiction. Hence \(m\) has no prime divisors other than \(2,3,13\), and therefore \(m\) divides \(2\cdot 3\cdot 13=78\).
\end{proof}

\begin{proposition} \label{P:7-large}
Let \(E/\Q\) be a non-CM elliptic curve with a primitive congruence obstruction modulo a squarefree integer \(m\). 
Suppose that \(7\) is the largest prime divisor of \(m\). Then \(m\) divides \(210\) and \(G_E(7)\neq\GL_2(\F_7)\). Further, if \(5\) divides \(m\), then the following hold:
\begin{itemize}
    \item\(G_E(2)=\GL_2(\F_2)\),
    \item \(G_E(3)\) is not conjugate to a subgroup of \(B_0(3)\),
    \item \(G_E(5)\) is not equal to \(\GL_2(\F_5)\) and is not conjugate to a subgroup of $B_0(5)$, and
    \item \(G_E(7)\) is not conjugate to a subgroup of $B_0(7).$
\end{itemize}
\end{proposition}

\begin{proof}
Since \(m\) is squarefree and \(7\) is its largest prime divisor, \(m\) divides \(2\cdot3\cdot5\cdot7=210\). We first note that the mod \(7\) image is not full. If \(m=7\), then \(G_E(7)\subseteq\mathcal I_1(7)\). If \(m\) is composite, then the properness of $G_E(7)$ follows immediately from Corollary~\ref{C:full-prime-support}. Thus in either case,
\(
    G_E(7)\neq\GL_2(\F_7).
\)

For the remainder of the proof, suppose that \(5\) divides \(m\). If \(G_E(5)=\GL_2(\F_5)\), then Corollary~\ref{C:full-prime-support} gives a prime \(p>5\) dividing \(m\) such that \(5\) divides the order of \(G_E(p)\). The only possible value of \(p\) is \(7\). On the other hand, \(|G_E(7)|\) divides \(|\GL_2(\F_7)|\), and
\[
    |\GL_2(\F_7)|=(7^2-1)(7^2-7)=2^5\cdot3^2\cdot7
\]
is not divisible by \(5\), a contradiction. Hence
\(
    G_E(5)\neq\GL_2(\F_5).
\)

Suppose next that \(G_E(5)\) is conjugate to a subgroup of \(B_0(5)\). Equivalently, \(E\) admits a rational \(5\)-isogeny. By \cite[Proposition~20(4)]{2512.00652}, the \(7\)-adic representation of \(E\) is then surjective, contradicting the properness of \(G_E(7)\). Thus \(G_E(5)\) is not conjugate to a subgroup of \(B_0(5)\). Similarly, if \(G_E(7)\) were conjugate to a subgroup of \(B_0(7)\), then \(E\) would admit a rational \(7\)-isogeny, and \cite[Proposition~21(4)]{2512.00652} would imply that the \(5\)-adic representation of \(E\) is surjective. This contradicts the properness of \(G_E(5)\). Hence \(G_E(7)\) is not conjugate to a subgroup of \(B_0(7)\).

We next consider the mod \(3\) image. Suppose that \(G_E(3)\) is conjugate to a subgroup of \(B_0(3)\), so that \(E\) admits a rational \(3\)-isogeny. We have already shown that \(E\) admits neither a rational \(5\)-isogeny nor a rational \(7\)-isogeny. If $E$ also admitted a rational $q$-isogeny for a prime $q>7$, then $E$ would in addition admit a rational cyclic $3q$-isogeny, contrary to Kenku's classification of rational cyclic isogenies \cite[Theorem~1]{MR675184}. Thus \(E\) admits no rational isogeny of prime degree greater than \(3\). By \cite[Proposition~19(4)]{2512.00652}, its \(7\)-adic representation is surjective, again contradicting the properness of \(G_E(7)\). Therefore \(G_E(3)\) is not conjugate to a subgroup of \(B_0(3)\).

Finally, we determine the mod \(2\) image. Suppose that \(G_E(2)\neq\GL_2(\F_2)\). If $G_E(2)$ is conjugate to a subgroup of $B_0(2)$, then $E$ admits a rational $2$-isogeny. We have already ruled out rational isogenies of degrees $3$, $5$, and $7$. If $E$ also admitted a rational $q$-isogeny for a prime $q>7$, then it would admit a rational cyclic $2q$-isogeny, again contradicting \cite[Theorem~1]{MR675184}. Hence $E$ would admit no rational isogeny of prime degree greater than $2$, and \cite[Proposition~18(3)]{2512.00652} would imply that its $7$-adic representation is surjective. This is impossible. Since \(\GL_2(\F_2)\cong S_3\), every proper subgroup of
\(\GL_2(\F_2)\) has order \(1\), \(2\), or \(3\), and every
subgroup of order \(1\) or \(2\) is contained in a Borel subgroup.
It follows that \(G_E(2)\) is the unique subgroup of order \(3\),
denoted \(G_{3,2}\) in \cite{MR3957898}.  Since \(G_E(7)\neq\GL_2(\F_7)\), \cite[Theorem~C(2)]{MR3957898} then implies that \(G_E(14)\) is conjugate to a subgroup of \(G_{3,2}\times G_{7,7}\), where \(G_{7,7}=B_0(7)\) in the notation of \cite{MR3957898}. But then \(G_E(7)\) is conjugate to a subgroup of \(B_0(7)\), contradicting what we proved above. Thus
\(
    G_E(2)=\GL_2(\F_2).
\)
\end{proof}

\begin{corollary}\label{C:candidate-levels}
Let \(E/\Q\) be a non-CM elliptic curve with a primitive congruence obstruction modulo a composite integer \(m\geq2\). Then
\begin{equation}\label{E:candidate-levels}
m\in\{6,10,14,15,21,26,30,35,39,42,70,78,105,210\}.
\end{equation}
\end{corollary}

\begin{proof}
The integer \(m\) is squarefree. By Proposition~\ref{P:sqfree2}, its largest prime divisor belongs to \(\{3,5,7,13\}\). If the largest prime divisor is \(13\), then Proposition~\ref{P:13-large} gives \(m\mid78\), and hence \(m\in\{26,39,78\}\). If the largest prime divisor is \(7\), then Proposition~\ref{P:7-large} gives \(m\mid210\), and hence
\[
m\in\{14,21,35,42,70,105,210\}.
\]
If the largest prime divisor is \(5\), then squarefreeness gives \(m\in\{10,15,30\}\), while if it is \(3\), then \(m=6\). Combining these possibilities gives the stated list.
\end{proof}

We conclude the section by obtaining additional restrictions at level \(m = 210\). These substantially reduce the number of prime level image tuples that must be considered for this level in the group-theoretic search.

\begin{proposition}\label{P:level-210}
Let \(E/\Q\) be a non-CM elliptic curve with a primitive congruence obstruction modulo \(210\). Then its tuple of mod \(2,3,5,\) and \(7\) images is one of the following:
\[
(\mathtt{2GL},\mathtt{3GL},\mathtt{5Nn},\mathtt{7Nn}), \;
(\mathtt{2GL},\mathtt{3GL},\mathtt{5S4},\mathtt{7Nn}), \;
(\mathtt{2GL},\mathtt{3GL},\mathtt{5S4},\mathtt{7Ns}), \;
(\mathtt{2GL},\mathtt{3Nn},\mathtt{5Nn},\mathtt{7Nn}),
\]
where we use the notation of \cite[Section 6.4]{MR3482279} and write $\mathtt{\ell} \mathtt{GL}$ for $\GL_2(\F_\ell)$.
\end{proposition}
\begin{proof}
For each $p$, write \(H_p\) for the label of the mod \(p\) image of $E$. Inspection of the groups appearing in Theorem~\ref{T:SOTA}, together with Proposition~\ref{P:7-large}, reveals that 
\[
H_2=\mathtt{2GL},\qquad
H_3\in\{\mathtt{3GL},\mathtt{3Nn},\mathtt{3Ns}\},
\]
\[
H_5\in\{\mathtt{5Nn},\mathtt{5S4},\mathtt{5Ns},\mathtt{5Ns.2.1}\},
\qquad
H_7\in\{\mathtt{7Nn},\mathtt{7Ns},\mathtt{7Ns.2.1},\mathtt{7Ns.3.1}\}.
\]
This is checked in the file \texttt{Level210.m}, along with the rest of the computations of the proof.

The mod \(7\) images \(\mathtt{7Ns.2.1}\) and \(\mathtt{7Ns.3.1}\) cannot occur under the present hypotheses. Indeed, the only non-CM \(j\)-invariant arising in these cases is
\(
j=\frac{2268945}{128}
\)
by \cite[Theorem~2]{MR2950703} and \cite[Remark~6.5]{MR3482279}. For a representative elliptic curve with this \(j\)-invariant, the accompanying \texttt{Magma} computation verifies that the mod \(5\) image is \(\GL_2(\F_5)\). Surjectivity of the mod \(5\) representation is preserved under quadratic twisting by \cite[Proposition~21]{MR4458130}, so every elliptic curve over \(\Q\) with this \(j\)-invariant has surjective mod \(5\) image. This contradicts Proposition~\ref{P:7-large}, which shows that \(G_E(5)\) must be proper.

Suppose now that \(H_7=\mathtt{7Ns}\). The modular curves associated with
\[
\mathtt{3Nn}\times\mathtt{7Ns}
\qquad\text{and}\qquad
\mathtt{3Ns}\times\mathtt{7Ns}
\]
have no non-CM rational points by \cite{MayleRouse,MayleRousecode}, so \(H_3=\mathtt{3GL}\).  By \cite[Theorem~6.2]{MR5049284}, the case \(H_5=\mathtt{5Nn}\) is impossible here, and the modular curve associated with
\(
\mathtt{5Ns}\times\mathtt{7Ns}
\)
has no non-CM rational points again by \cite{MayleRouse,MayleRousecode}. Since \(\mathtt{5Ns.2.1}\) is a subgroup of \(\mathtt{5Ns}\), this also excludes \(H_5=\mathtt{5Ns.2.1}\). From these considerations, it follows that
\[
(H_3,H_5,H_7)=(\mathtt{3GL},\mathtt{5S4},\mathtt{7Ns}).
\]

It remains to consider \(H_7=\mathtt{7Nn}\). By \cite[Theorem~6.2]{MR5049284}, we have that
\[
H_3\in\{\mathtt{3GL},\mathtt{3Nn}\},
\qquad
H_5\in\{\mathtt{5Nn},\mathtt{5S4}\}.
\]
The pair 
\(
(H_3, H_5) = (\mathtt{3Nn}, \mathtt{5S4})
\)
cannot occur as a pair of mod $3$ and mod $5$ Galois images for an elliptic curve over $\Q$ by \cite[Section~5]{MR4458130}. The remaining three possibilities are therefore
\[
(\mathtt{3GL},\mathtt{5Nn},\mathtt{7Nn}),\qquad
(\mathtt{3GL},\mathtt{5S4},\mathtt{7Nn}),\qquad
(\mathtt{3Nn},\mathtt{5Nn},\mathtt{7Nn}),
\]
which, together with the case \(H_7=\mathtt{7Ns}\), give the four tuples in the statement.
\end{proof}

\section{The finite group-theoretic search}\label{S:Search}

In this section, we describe the finite group-theoretic computation. By Corollary~\ref{C:candidate-levels}, only the fourteen levels in \eqref{E:candidate-levels} need to be considered. We begin with a simple, computationally inexpensive necessary condition on the prime level images.

\begin{lemma} \label{L:Sum}
    Let $m = p_1 \cdots p_r$ be a squarefree composite integer, where the $p_i$ are distinct primes. If $G \subseteq \GL_2(\Z/m\Z)$ has a congruence obstruction modulo $m$, then
\[ \sum_{i=1}^r \mathcal{F}_1(G(p_i)) \geq 1,\]
where $\mathcal{F}_1$ was defined in \eqref{eq:F1}.
\end{lemma}
\begin{proof}
For each \(1\leq i\leq r\), write \(G_i\coloneqq G(p_i)\), let
\(
\pi_i\colon G\to G_i
\)
be the reduction map, and set
\[
S_i\coloneqq\{g\in G:\pi_i(g)\in\mathcal I_1(p_i)\}.
\]
Since \(\pi_i\) is surjective, all of its fibers have the same cardinality, so
\(
|S_i|/|G| = \mathcal F_1(G_i).
\)
By the Chinese remainder theorem,
\[
g\in\mathcal I_1(m)
\quad\Longleftrightarrow\quad
\pi_i(g)\in\mathcal I_1(p_i)
\quad\text{for some }i.
\]
The obstruction condition therefore implies that $G = \bigcup_{i=1}^r S_i$. Taking cardinalities and dividing through by $|G|$ proves the result.
\end{proof}

Fix a level \(m=p_1\cdots p_r\) from \eqref{E:candidate-levels}. We consider  tuples $(G_1, \ldots, G_r)$ of representatives, up to conjugacy, of the possible mod $p_i$ images in Theorem~\ref{T:SOTA}, including $\GL_2(\F_{p_i})$. We narrow down to only those which satisfy 
\[
\sum_{i=1}^r \mathcal F_1(G_i)\geq 1
\quad \text{and} \quad
\mathcal F_1(G_i)<1 \; \text{for each $1\leq i \leq r$} \]
and which are not ruled out by the restrictions of Section~\ref{S:prime-restrictions}. For each such tuple \((G_1,\ldots,G_r)\), we set
\(
H_0\coloneqq G_1\times\cdots\times G_r,
\)
which we view as a subgroup of \(\GL_2(\mathbb Z/m\mathbb Z)\) via the isomorphism of the Chinese remainder theorem. Starting at $H_0$,  we descend through maximal subgroups (again, up to conjugacy). In this process, a group $H$ is retained for further consideration only if it satisfies all four of the following conditions:
\begin{enumerate}
    \item $H(p_i) = G_i$ for each $1 \leq i \leq r$,
    \item  $\det H = (\Z/m\Z)^\times$,
    \item $H$ contains a matrix $c$ with $c^2 = I$, $\tr c = 0$, and $\det c = -1$, and
    \item $H$ has no congruence obstruction modulo any proper divisor $n \geq 2$ of $m$.
\end{enumerate}
Notice that for a non-CM elliptic curve over $\Q$ with a primitive obstruction modulo $m$, the mod $m$ image satisfies all of these conditions, as does every supergroup of the mod $m$ image contained in $H_0$ since a failure of any of the four conditions persists under passage to any subgroup.

When $H\subseteq\mathcal I_1(m)$, the fourth condition makes $H$ a primitive obstruction group. We record the pair $(m,H)$ and consider the coarse modular curve $X_{H^\pm}$, where $H^\pm=\langle H,-I\rangle$. If $X_{H^\pm}(\Q)$ is known to be finite, descent below $H$ can stop: every elliptic curve over $\Q$ whose mod $m$ image is contained in $H$, up to conjugacy, has its $j$-invariant among the rational $j$-invariants of this curve. On the other hand, if $X_{H^{\pm}}(\Q)$ is infinite or we are unable to determine whether it is infinite (in principle, this can occur for a genus 1 curve), the search continues through the subgroups of $H$. 

After completing the search at a fixed level \(m\), we keep only those recorded groups \(H\) with \(X_{H^\pm}(\Q)\) known to be finite that are maximal under containment up to conjugacy in \(\GL_2(\Z/m\Z)\) among such recorded groups. Indeed, if \(K\) is contained in a conjugate of a strictly larger such group \(H\), then \(X_{K^\pm}\) maps to \(X_{H^\pm}\) over \(\Q\), compatibly with the \(j\)-maps. Every possible mod \(m\) image contained in \(K\), up to conjugacy, is therefore already covered by \(H\). This reduction is applied across all local tuples at level \(m\). Groups for which the coarse curve has infinitely many rational points or its finiteness is unresolved are retained.

This procedure is exhaustive for images of non-CM elliptic curves over $\Q$ with a primitive obstruction: up to conjugacy, every such image lies in a chain of maximal subgroups below one of the $H_0$ in which every group on that chain passes the four necessary conditions. Thus either the image itself is reached or the corresponding $j$-invariant arises from a rational point on a coarse modular curve with finitely many rational points. For many of the groups $H_0$ we encounter, we find no subgroup that satisfies the four necessary conditions above and has a congruence obstruction modulo $m$.

We summarize the results of this group-theoretic search in the following proposition.

\begin{proposition}\label{P:search-output}
Let \(E/\Q\) be a non-CM elliptic curve with a primitive congruence obstruction modulo a composite integer \(m\). Then either \(G_E(m)\) is conjugate to one of the groups listed in Table~\ref{tab:infpoints}, or $j(E)$ arises from a rational point on one of the coarse curves listed with obstruction modulus $m$ in Table~\ref{tab:finitepoints}.
\end{proposition}

\section{Rational points on the remaining modular curves} \label{S:RatPtWork}

The search described in Section \ref{S:Search} determines eight primitive obstruction groups (up to conjugacy) whose coarse modular curves have infinitely many rational points. Table \ref{tab:infpoints} gives each of these groups (specified by LMFDB label), together with the level of the obstruction $m$, index, and an example of an elliptic curve over $\Q$ (also specified by LMFDB label) whose mod $m$ Galois image is the group.

\begin{table}[H]
\centering
\caption{Primitive obstruction images with infinitely many rational points}
\label{tab:infpoints}
\begin{tabular}{|c|c|l|l|}
\hline
\(m\) & Index & Group & Example \\ \hline

6  & 24   & \mc{6.24.0-3.a.1.1} &  \ec{2116800.cri3} \\

6 & 8    &  \mc{6.8.0-3.a.1.1} & \ec{129600.ip2} \\ 
6  & 8    & \mc{6.8.0-3.a.1.2} &   \ec{254016.ee1} \\  
10 & 120  & \mc{10.120.0-5.a.1.1}& \ec{121.d2}\\ 
10 & 24   & \mc{10.24.0-5.a.2.2} & \ec{1225.i2}\\ 
10 & 24   & \mc{10.24.0-5.a.1.1}& \ec{1600.i4} \\ 
14 & 48   &  \mc{14.48.0-7.a.1.1} &  \ec{10816.c2} \\ 
14 & 48   & \mc{14.48.0-7.a.2.1} &  \ec{28224.co2} \\ \hline

\end{tabular}
\end{table}

For every group $H$ in Table~\ref{tab:infpoints}, the coarse curve $X_{H^\pm}$ has genus zero, and the displayed example corresponds to a rational point on $X_{H^\pm}$. Thus $X_{H^\pm} \cong \mathbb{P}^1$ over $\Q$, and Hilbert's irreducibility gives infinitely many distinct non-CM rational $j$-invariants for which $\langle G_E(m), -I \rangle$ is conjugate to $H^{\pm}$; see \cite[Lemma~3.5]{zywina2015possibleimagesmodell}. Since \(-I\notin H\), after an appropriate quadratic twist the mod \(m\) Galois image is conjugate to \(H\); see \cite[Corollary~5.25]{MR3482279}. This justifies the assertion made following the statement of Theorem~\ref{T:class}.

Next, the search gives $28$ groups (up to conjugacy) whose coarse modular curves have finitely many rational points, and are maximal with this property (as explained in Section~\ref{S:Search}). These groups give 16 distinct coarse curves. In Table~\ref{tab:finitepoints}, we list these coarse curves together with the obstruction modulus $m$, the number of distinct non-CM rational $j$-invariants, and the method used to determine them. The label ``Genus $2$ rank $0$'' means the coarse modular curve has genus $2$ and its Jacobian has rank $0$, and we used the \texttt{Chabauty0} command in Magma to compute its rational points. The label ``Maps to rank $0$ EC'' means the coarse curve admits a nonconstant morphism to a rank $0$ elliptic curve over $\Q$; its rational $j$-invariants were determined in \cite{MayleRouse} and we obtain them from the file \texttt{allpointcounts.txt} from the repository \cite{MayleRousecode}. Finally, ``Rank $0$ EC'' means the coarse curve is a rank $0$ elliptic curve and we directly computed its rational points, which are its torsion points, in \texttt{Magma}. 

\begin{table}[H]
\centering
\caption{Modular curves requiring rational point computations}
\label{tab:finitepoints}
\begin{tabular}{|c|l|c|l|} \hline
$m$  & Coarse curve & \shortstack{No. of non-CM\\$j$-invariants} & Method \\ \hline
10  & \mc{10.24.1.b.1} & 0 & Rank $0$ EC  \\
10  & \mc{10.24.1.b.2} & 0 & Rank $0$ EC  \\
14  & \mc{14.48.2.e.1} & 0 & Genus $2$ rank $0$ \\ 
14  & \mc{14.48.2.e.2} & 0  & Genus $2$ rank $0$ \\
15, 30 & \mc{15.48.1.a.1} &2 & Rank $0$ EC \\ 
15, 30 & \mc{15.48.1.a.2} & 2 & Rank $0$ EC \\ 
15, 30  & \mc{15.72.3.e.1}& 0 & Maps to rank $0$ EC  \\ 
15, 30 & \mc{15.72.3.e.2} & 0 & Maps to rank $0$ EC \\ 
21, 42 &\mc{21.96.3.a.1} & 0 & Maps to rank $0$ EC  \\ 
21, 42 & \mc{21.96.3.a.2}& 0 & Maps to rank $0$ EC  \\ 
21, 42 &\mc{21.144.7.b.1}& 0  &Maps to rank $0$ EC  \\ 
21, 42 & \mc{21.144.7.b.2} & 0 & Maps to rank $0$ EC  \\ 
30 & \mc{30.72.5.bn.1}&0 &Maps to rank $0$ EC  \\ 
30 & \mc{30.72.5.bn.2}& 0 & Maps to rank $0$ EC  \\
42 & \mc{42.144.10.t.1}& 0&Maps to rank $0$ EC  \\ 
42 & \mc{42.144.10.t.2}& 0& Maps to rank $0$ EC\\ 
\hline
\end{tabular}
\end{table}

We observe that the candidate levels \(21\) and \(42\) produce no non-CM elliptic
curves. Of the \(16\) rows in Table~\ref{tab:finitepoints}, \(14\) have
no non-CM rational points on their associated coarse curves. We now treat the two remaining rows, which involve two distinct coarse curves.

\begin{proposition}\label{P:exceptional-twists}
Let $m \in \{15,30\}$. The non-CM elliptic curves $E/\Q$ with a primitive congruence obstruction modulo $m$ whose $j$-invariant arises from a rational point on one of the two remaining coarse curves in Table~\ref{tab:finitepoints} are precisely the curves listed with level $m$ in Table~\ref{tab:exceptional-images}.
\end{proposition}

\begin{proof}
Each of the two coarse curves $X_K$ has eight rational points with exactly two non-CM rational \(j\)-invariants. Choose a representative \(E_j/\Q\) for each of these \(j\)-invariants. Every elliptic curve over \(\Q\) with the same non-CM \(j\)-invariant is a quadratic twist of \(E_j\). From LMFDB we know that primes $7$ and $13$ do not divide adelic level of any $E_j$. Therefore, their mod $7$ and mod $13$ Galois images are $\GL_2(\Z/7\Z)$ and $\GL_2(\Z/13\Z)$, respectively. Hence, any of their quadratic twists will also have surjective mod $7$ and mod $13$ Galois images. From Proposition \ref{P:13-large} and Proposition \ref{P:7-large}, we know that any of their quadratic twists will have primitive obstruction mod $15$ or mod $30$.

We now justify that only finitely many twists need to be tested. For any two distinct quadratic twists of a fixed $E_j$, their mod $15$ (or mod $30$) Galois image would produce the same coarse modular curve. Therefore, it suffices to look at twists whose image does not contain $-I.$ From LMFDB we know that that for any $E_j$, their mod $15$ (or mod $30$) Galois images do not contain $-I$. Suppose that $\chi$ is a quadratic character for which
\[
-I \notin \chi\rho_{E_j,m}(\Gal(\overline{\Q}/\Q)).
\]
Since \(-I \notin \rho_{E_j,m}(\Gal(\overline{\Q}/\Q))\), this forces \(\chi(\sigma)=1\). Thus \(\chi\) factors through \(\Gal(\Q(E_j[m])/\Q)\), which proves that the enumeration of quadratic characters of the division field, including the trivial character, is exhaustive.

The computation in \texttt{lemma rational points.m} enumerates these characters, tests the trivial twist and every resulting nontrivial twist, checks primitivity at modulus $m$, and determines the Galois image in each case. It gives precisely the eight curves and images in Table~\ref{tab:exceptional-images}. For example, for \(\ec{50.a2}\), the relevant quadratic fields are 
\[\Q(\sqrt{-2}), \Q(\sqrt{-3}),
\Q(\sqrt{5}), \Q(\sqrt6), \Q(\sqrt{-10}), \Q(\sqrt{-15}), \Q(\sqrt{30}) 
\]
and the eight twists, including the trivial twist, leave only \(\ec{450.c3}\) and \(\ec{14400.ef3}\) for the two primitive obstruction groups. The other three cases are analogous.
\end{proof}

The Galois images of these eight curves are recorded in the following table.

\begin{table}[H]
\centering
\caption{Galois images at levels \(15\) and \(30\)}
\label{tab:exceptional-images}
\begin{tabular}{|c|l|l|}
\hline
\(m\) & \(E\) & LMFDB label of Galois image modulo $m$\\ \hline
15 & \ec{450.c1} & \mc{15.192.1-15.a.1.4} \\
15 & \ec{450.c2} & \mc{15.192.1-15.a.2.2} \\
15 & \ec{450.c3} & \mc{15.192.1-15.a.4.1} \\
15 & \ec{450.c4} & \mc{15.192.1-15.a.3.3} \\
30 & \ec{14400.ef1} & \mc{30.192.1-15.a.1.4} \\
30 & \ec{14400.ef2} & \mc{30.192.1-15.a.2.4} \\
30 & \ec{14400.ef3} & \mc{30.192.1-15.a.4.2} \\
30 & \ec{14400.ef4} & \mc{30.192.1-15.a.3.2} \\ \hline
\end{tabular}
\end{table}

\begin{proof}[Proof of Theorem~\ref{T:class}]
Suppose first that $E$ has a primitive congruence obstruction modulo a composite integer $m \geq 2$. By Proposition~\ref{P:search-output}, the group \(G_E(m)\) is either conjugate to one of the groups in Table~\ref{tab:infpoints}, or $j(E)$ arises from a rational point on one of the coarse curves listed with obstruction modulus $m$ in Table~\ref{tab:finitepoints}. The rational point computations in Table~\ref{tab:finitepoints} eliminate every row except the two rows with coarse curves \mc{15.48.1.a.1} and \mc{15.48.1.a.2}, each occurring at obstruction moduli $15$ and $30$. In particular, this eliminates the finite rational point cases at moduli $10$ and $14$, as well as all cases at moduli $21$ and $42$. Proposition~\ref{P:exceptional-twists} determines the twists arising from these two coarse curves at moduli $15$ and $30$. Thus
\[
m \in \{6,10,14,15,30\}
\]
and \(G_E(m)\) is, up to conjugacy, one of the groups in Tables~\ref{tab:infpoints} and \ref{tab:exceptional-images}.

For each group \(H\) listed with obstruction modulus $m$ in Tables~\ref{tab:infpoints} and \ref{tab:exceptional-images}, we verify in \texttt{Magma} that
\[
\mathcal F_1(H)=1
\qquad\text{and}\qquad
\mathcal F_1(H(n))<1\quad\text{for every proper divisor $n \geq 2$ of $m$}.
\]
Hence the corresponding obstruction is primitive. This proves the converse.
\end{proof}

\section{The primitive obstruction set and universal level} \label{S:set}

We now determine which primitive obstruction levels can occur simultaneously. We first record the invariance of the primitive obstruction set under \(\Q\)-isogeny and determine the possible prime obstruction levels. We then rule out simultaneous composite obstructions and the simultaneous occurrence of a prime obstruction with a primitive composite obstruction.

\begin{lemma}\label{L:isogeny-invariance}
If \(E\) and \(E'\) are \(\Q\)-isogenous elliptic curves, then
\[
\operatorname{PrimObs}(E)=\operatorname{PrimObs}(E').
\]
\end{lemma}

\begin{proof}
For all but finitely many primes $p$, the two reductions have the same
Frobenius trace and hence the same number of points, so $E$ has a congruence obstruction at any given level if and only if $E'$ does. Applying this also to the proper divisors of the level proves the result. 
\end{proof}

\begin{proposition}\label{P:prime-set}
As \(E/\Q\) varies over all non-CM elliptic curves, the possible values of
\[
\{ \ell : \ell \text{ is prime and } \ell \in \operatorname{PrimObs}(E) \}
\]
are as follows:
\[
\emptyset,\quad
\{2\},\quad
\{3\},\quad
\{5\},\quad
\{7\},\quad
\{2,3\},\quad
\{2,5\}.
\]
\end{proposition}

\begin{proof}
As noted in the proof of Proposition~\ref{P:prime-level}, \(E\) has a congruence obstruction modulo a prime \(\ell\) if and only if \(E\) is isogenous over \(\Q\) to an elliptic curve with a rational point of order \(\ell\). The classification of isogeny-torsion graphs over \(\Q\) in \cite[Theorem~1.3 and Tables~1--4]{MR4203041} then gives the displayed possibilities.
\end{proof}

We next rule out the simultaneous occurrence of two primitive composite obstructions. All computational assertions in the proof are verified in \texttt{multipleobstructions.m}.

\begin{proposition} \label{P:no-two-composite}
Let $E/\Q$ be a non-CM elliptic curve. There is at most one composite integer $m \geq 2$ such that $E$ has a primitive congruence obstruction modulo $m$.
\end{proposition}
\begin{proof}
By Theorem~\ref{T:class}, the only possible composite primitive obstruction levels are \(6, 10, 14, 15,\) and \(30\). The computation of the adelic images of the eight curves in Table~\ref{tab:exceptional-images} shows that none has a second primitive
composite obstruction. It remains to consider simultaneous obstructions arising from two distinct levels among \(6,10,14\).

There are \(21\) pairs of Galois image types. For the \(17\) pairs in Table~\ref{tab:multiple-obstructions}, the corresponding fiber product modular curve covers a modular curve having no non-CM rational points. The covering maps arise from subgroup containments, up to conjugacy, after reduction to the level of the covered curve and adjoining $-I$. The rational point determinations come from \texttt{allpointcounts.txt} of \cite{MayleRousecode}. 

\begin{table}[H]
\centering
\caption{Ruling out multiple congruence obstructions}
\label{tab:multiple-obstructions}
\begin{tabular}{|l|l|c|}
\hline
Pair & Covered curve  & \shortstack{No. of non-CM\\$j$-invariants}  \\ \hline
(\mc{6.24.0-3.a.1.1}, \mc{10.120.0-5.a.1.1})  & \mc{15.72.3.a.1} & 0 \\
(\mc{6.24.0-3.a.1.1}, \mc{10.24.0-5.a.2.2})  & \mc{15.72.3.a.1}  &  0 \\
(\mc{6.24.0-3.a.1.1}, \mc{10.24.0-5.a.1.1})  & \mc{15.72.3.a.1}  & 0 \\
(\mc{6.24.0-3.a.1.1}, \mc{14.48.0-7.a.1.1}) & \mc{21.96.5.a.1}  & 0 \\
(\mc{6.24.0-3.a.1.1}, \mc{14.48.0-7.a.2.1}) & \mc{21.96.5.a.1}   &  0 \\
(\mc{6.8.0-3.a.1.1}, \mc{10.120.0-5.a.1.1}) & \mc{15.240.9.a.1}  & 0 \\
(\mc{6.8.0-3.a.1.1}, \mc{14.48.0-7.a.1.1}) & \mc{21.96.3.a.1}  & 0 \\
(\mc{6.8.0-3.a.1.1}, \mc{14.48.0-7.a.2.1}) & \mc{21.96.3.a.2}  & 0 \\
(\mc{6.8.0-3.a.1.2}, \mc{10.120.0-5.a.1.1}) & \mc{15.240.9.a.1}  &  0 \\
(\mc{6.8.0-3.a.1.2}, \mc{14.48.0-7.a.1.1}) & \mc{21.96.3.a.1}   &  0 \\
(\mc{6.8.0-3.a.1.2}, \mc{14.48.0-7.a.2.1}) & \mc{21.96.3.a.2}   &  0 \\
(\mc{10.120.0-5.a.1.1}, \mc{14.48.0-7.a.1.1}) & \mc{35.144.7.a.2} & 0 \\
(\mc{10.120.0-5.a.1.1}, \mc{14.48.0-7.a.2.1}) & \mc{35.144.7.a.1} & 0 \\
(\mc{10.24.0-5.a.2.2}, \mc{14.48.0-7.a.1.1}) & \mc{35.144.7.a.2} & 0 \\
(\mc{10.24.0-5.a.2.2}, \mc{14.48.0-7.a.2.1}) & \mc{35.144.7.a.1}  & 0 \\
(\mc{10.24.0-5.a.1.1}, \mc{14.48.0-7.a.1.1}) & \mc{35.144.7.a.2}  & 0 \\
(\mc{10.24.0-5.a.1.1}, \mc{14.48.0-7.a.2.1}) & \mc{35.144.7.a.1}  & 0 \\ \hline
\end{tabular}
\end{table}

The remaining four pairs are
\[
\begin{aligned}
&(\mc{6.8.0-3.a.1.1},\mc{10.24.0-5.a.2.2}),&
&(\mc{6.8.0-3.a.1.1},\mc{10.24.0-5.a.1.1}),\\
&(\mc{6.8.0-3.a.1.2},\mc{10.24.0-5.a.2.2}),&
&(\mc{6.8.0-3.a.1.2},\mc{10.24.0-5.a.1.1}).
\end{aligned}
\]
Each corresponding fiber product covers \(X_0(15)\), whose four non-CM rational \(j\)-invariants are represented by
\[
\ec{50.a1},\quad \ec{50.a2},\quad \ec{50.a3},\quad \ec{50.a4}.
\]
Neither \(\mc{6.8.0-3.a.1.1}\) nor \(\mc{6.8.0-3.a.1.2}\) contains \(-I\). The kernel argument used in the proof of Proposition~\ref{P:exceptional-twists} therefore shows that, for any representative $E$, a quadratic character giving a twist whose mod \(6\) image is contained in one of these groups must factor through \(\Q(E[6])/\Q\). For each of the four representatives, the computation enumerates the quadratic characters factoring through its \(6\)-division field and tests every corresponding twist against all four pairs. No twist has its mod \(6\) and mod \(10\) images contained in the respective groups. Hence no non-CM elliptic curve has two primitive composite congruence obstructions.
\end{proof}

\begin{proposition}\label{P:no-prime-composite}
Let \(E/\Q\) be a non-CM elliptic curve with a primitive congruence obstruction modulo a composite integer \(d\). Then \(E\) has no congruence obstruction modulo any prime \(q\).
\end{proposition}

\begin{proof}
By Theorem~\ref{T:class},
\(
d\in\{6,10,14,15,30\}.
\)
Suppose that \(E\) also has a congruence obstruction modulo a prime \(q\). Since the obstruction modulo \(d\) is primitive, we have \(q\nmid d\).

The groups in Tables~\ref{tab:infpoints} and \ref{tab:exceptional-images} show that a primitive obstruction modulo \(6\), \(10\), or \(14\) forces \(E\) to admit a rational isogeny of degree \(3\), \(5\), or \(7\), respectively, while a primitive obstruction modulo \(15\) or \(30\) forces \(E\) to admit a rational cyclic \(15\)-isogeny. On the other hand, Proposition~\ref{P:prime-level} implies that \(E\) is isogenous over \(\Q\) to an elliptic curve with a rational point of order \(q\). The classification of isogeny-torsion graphs \cite[Theorem~1.3 and Tables~1--4]{MR4203041} then leaves only
\[
(d,q)\in \{ (6,5), (10,3), (14,3)\}.
\]
The modular curve computations in \cite[Theorem~17 and Table~12]{MR4458130} determine the elliptic curves that can arise in these three cases. Up to \(\Q\)-isogeny, the possibilities are \ec{50.b} for \((6,5)\), \ec{50.a} and \ec{450.g} for \((10,3)\), and \ec{162.b} and \ec{162.c} for \((14,3)\).

For each of these five isogeny classes, we verify in the file \texttt{multipleobstructions.m} that
\[
\mathcal F_1(G_E(d))<1
\]
for one representative \(E\) of each class, using the corresponding value of \(d\). Hence that representative does not have a congruence obstruction modulo \(d\). By the proof of Lemma~\ref{L:isogeny-invariance}, no curve in any of these isogeny classes has a congruence obstruction modulo \(d\), which is a contradiction.
\end{proof}

\begin{proof}[Proof of Theorem~\ref{T:set}]
If \(E\) has no primitive composite obstruction, Proposition~\ref{P:prime-set} gives exactly the possibilities consisting only of prime levels. If \(E\) has a primitive composite obstruction modulo \(d\), then Theorem~\ref{T:class} gives
\[
d\in\{6,10,14,15,30\}.
\]
Propositions~\ref{P:no-two-composite} and \ref{P:no-prime-composite} show that no other primitive obstruction can occur, so \(\operatorname{PrimObs}(E)=\{d\}\). This gives the complete displayed list in Theorem~\ref{T:set}. Finally, in \texttt{multipleobstructions.m} we compute an elliptic curve realizing each set in the list, proving that every possibility occurs.
\end{proof}

\begin{proof}[Proof of Corollary~\ref{C:universal210}]
Suppose that \(E\) has a congruence obstruction modulo some integer \(m\geq2\). It then has an obstruction modulo \(\rad(m)\), and among the divisors \(d\geq2\) of \(\rad(m)\) for which an obstruction occurs, choose one that is minimal under divisibility. The obstruction modulo \(d\) is primitive. By Theorem~\ref{T:set}, we have \(d\mid210\), and hence an obstruction modulo \(d\) gives an obstruction modulo \(210\). The converse is immediate by taking \(m=210\).
\end{proof}

\section{A level 30 primitive congruence obstruction} \label{S:NumEx}

In this section, we examine a representative of the isogeny class $\ec{14400.ef}$ and study, through Galois representations and division fields, the mechanism underlying the primitive congruence obstruction modulo $30$. This example is particularly noteworthy because its division fields exhibit several entanglements. The explanation is analogous to that of Jones’s example in \cite[Section~1.1]{MR2805578}.

Consider the elliptic curve
\[
E\colon y^2=x^3+12660x-248560.
\]
This curve has LMFDB label \(\ec{14400.ef4}\), and by
Table~\ref{tab:exceptional-images} it has a primitive congruence
obstruction modulo \(30\). We now give a direct explanation of this
obstruction in terms of the mod \(2\), mod \(3\), and mod \(5\)
representations.

Direct computation gives that the discriminant of the \(2\)-division
polynomial has squareclass \(-2\), and that
\begin{equation}\label{PQdefinition}
    Q=(-30,256\sqrt{-10})\in E[3],
\qquad
P=(34,192\sqrt6)\in E[5]
\end{equation}
have orders \(3\) and \(5\), respectively. Let $\varepsilon_2$, $\varepsilon_3$, and $\varepsilon_5$ denote the quadratic characters associated with $\Q(\sqrt{-2})$, $\Q(\sqrt{-10})$, and $\Q(\sqrt{6})$, respectively. For primes $p \nmid 30$, their values are
\[
\varepsilon_2(p) =\left(\frac{-2}{p}\right),
\qquad
\varepsilon_3(p) =\left(\frac{-10}{p}\right),
\qquad
\varepsilon_5(p) =\left(\frac{6}{p}\right).
\]
The coordinates of \(Q\) and \(P\) show that the lines
\(\langle Q\rangle\) and \(\langle P\rangle\) are Galois stable, with
Galois acting through \(\varepsilon_3\) and \(\varepsilon_5\),
respectively.

For \(\ell\in\{3,5\}\) and every prime \(p\nmid30N_E\), the
Galois-stable line above shows that
\(\varepsilon_\ell(p)\) is an eigenvalue of
\(\rho_{E,\ell}(\Frob_p)\). Since
\[
\det\rho_{E,\ell}(\Frob_p)\equiv p\pmod\ell,
\]
the other eigenvalue is \(p\varepsilon_\ell(p) \in \F_\ell\). It follows that
\[
3\nmid\#E_p(\F_p)
\quad\Longleftrightarrow\quad
p\equiv1\pmod3
\quad\text{and}\quad
\varepsilon_3(p)=-1,
\]
and
\[
5\nmid\#E_p(\F_p)
\quad\Longleftrightarrow\quad
p\not\equiv-1\pmod5
\quad\text{and}\quad
\varepsilon_5(p)=-1.
\]

These conditions depend only on the residue class of \(p\) modulo
\(120\). Using quadratic reciprocity to evaluate
\(\varepsilon_3(p)\) and \(\varepsilon_5(p)\) gives the following
divisibility pattern on the reduced residue classes modulo \(120\). A
triangle indicates divisibility by \(3\), a pentagon indicates
divisibility by \(5\), and both symbols indicate divisibility by both
\(3\) and \(5\). An unmarked residue indicates divisibility by neither $3$ nor $5$.
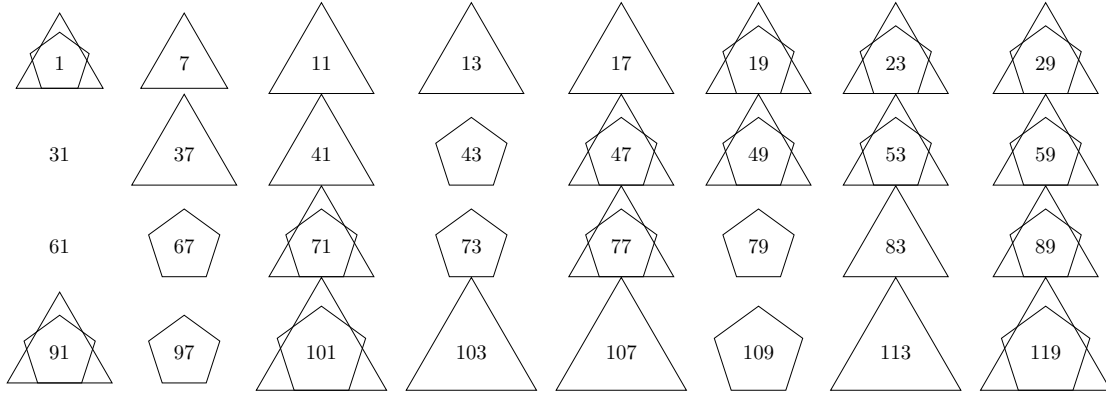
\begin{center}
\resizebox{.9\textwidth}{!}{$
\begin{array}{cccccccc}
\both{1} & \tri{7} & \tri{11} & \tri{13} & \tri{17} & \both{19} & \both{23} & \both{29} \\
\none{31} & \tri{37} & \tri{41} & \pent{43} & \both{47} & \both{49} & \both{53} & \both{59} \\
\none{61} & \pent{67} & \both{71} & \pent{73} & \both{77} & \pent{79} & \tri{83} & \both{89} \\
\both{91} & \pent{97} & \both{101} & \tri{103} & \tri{107} & \pent{109} & \tri{113} & \both{119}
\end{array}
$}
\end{center} \vspace{.25em}
In particular, if $p$ is congruent to neither $31$ nor $61$ mod $120$, then either \(3\mid\#E_p(\F_p)\) or \(5\mid\#E_p(\F_p)\).

It remains to consider the two residue classes \(31\) and \(61\)
modulo \(120\). For primes in either class,
\[
\varepsilon_2(p)=\left(\frac{-2}{p}\right)=-1.
\]
Since the discriminant of the \(2\)-division polynomial has squareclass
\(-2\), the character \(\varepsilon_2\) is the sign character for the permutation action on the three nonzero points of \(E[2]\). Thus
\(\rho_{E,2}(\Frob_p)\) acts as an odd permutation of these three
points. Every odd permutation of three elements is a transposition, so
\(\Frob_p\) fixes a nonzero point of \(E[2]\). Hence
\[
2\mid\#E_p(\F_p)
\]
for \(p\equiv31\) or \(61\pmod{120}\). Additionally, $E$ has no congruence obstruction modulo $2$, since $G_E(2) = \GL_2(\Z/2\Z)$.

We have therefore shown that
\[
\gcd(\#E_p(\F_p),30)>1
\]
    for every prime \(p\nmid30N_E\). Finally, direct computation gives
\[
\#E_{11}(\F_{11})=9,
\qquad
\#E_{31}(\F_{31})=34,
\qquad
\#E_{67}(\F_{67})=55.
\]
These orders are relatively prime to \(10\), \(15\), and \(6\),
respectively. There is therefore no congruence obstruction modulo
\(10\), \(15\), or \(6\). Since these are the maximal
proper divisors of \(30\), the congruence obstruction modulo \(30\) is
primitive.

We now investigate this phenomenon from the perspective of division fields. For each positive integer $n$, let $K_n = \Q(E[n])$ denote the $n$-th division field of $E$. For coprime positive integers $m$ and $n$, we say that $E$ has an $(m,n)$-entanglement if $K_m\cap K_n$ is strictly larger than $\Q$.

Recall that the discriminant of $E$ has squareclass $-2$. By \eqref{PQdefinition} and the Weil pairing, we have
\begin{equation}\label{quadraticfields}
\Q(\sqrt{-2}) \subseteq K_2,\qquad
\Q(\sqrt{-3},\sqrt{-10}) \subseteq K_3,\qquad
\Q(\sqrt{5},\sqrt{6}) \subseteq K_5.
\end{equation}
For coprime positive integers $a$ and $b$, we use the index data in the Galois representation sections of the LMFDB to check whether there is an $(a,b)$-entanglement. In particular, we find that
$$K_2 \cap K_3 = K_2 \cap K_5 = \Q \qquad \text{ and } \qquad K_3 \cap K_5 =\Q(\sqrt{30}).$$
From \eqref{quadraticfields} and the LMFDB data, we see that
$$K_2\cap K_{15} = \Q(\sqrt{-2}),\qquad
K_3\cap K_{10} = \Q(\sqrt{-3},\sqrt{-10}),\qquad
K_5\cap K_6=\Q(\sqrt{5},\sqrt{6}).$$
Thus, for each factorization $30=pqr$ into primes, the curve $E$ has a $(p,qr)$-entanglement.

\section{Least prime divisors in the absence of congruence obstructions} \label{S:ProofOfQuant}

In this section, we prove the quantitative version of Jones's unboundedness result stated in Theorem~\ref{Quanticep}. Our main tool is the following effective form of the Chebotarev density theorem with avoidance, proved in \cite[Corollary~6]{MR4743814} using work of Bach and Sorenson \cite{MR1355006}.

\begin{theorem} \label{ECDTA} Let $K/\Q$ be a finite Galois extension, let $m$ be a squarefree positive integer, and set $\widetilde{K} \coloneqq K(\sqrt{m})$. Assume GRH for the Dedekind zeta function of $\widetilde K$. Let $C \subseteq \Gal(K/\Q)$ be a nonempty subset that is closed under conjugation. Then there exists a prime $p \nmid m$ that is unramified in $K/\Q$ and whose Frobenius conjugacy class in $\Gal(K/\Q)$ is contained in $C$, such that
\begin{equation}
    p \ll  (\log d_{\widetilde{K}} + [\widetilde{K}:\Q])^2,
\end{equation} 
where $d_{\widetilde{K}}$ is the absolute discriminant of $\widetilde{K}$, and the implied constant is absolute.
\end{theorem}

We now prove Theorem~\ref{Quanticep}.

\begin{proof}[Proof of Theorem~\ref{Quanticep}]
Fix an integer \(n\ge 3\), and set
\[ M\coloneqq \prod_{\ell\le n}\ell. \]
Consider the subset
\[
C\coloneqq G_E(M) \setminus \mathcal{I}_1(M) = \{A\in G_E(M): \det(I-A)\in (\Z/M\Z)^\times\}\subseteq G_E(M).
\]
Since \(E\) has no congruence obstruction modulo any integer \(m\geq 2\), in particular, it has no congruence obstruction modulo \(M\). Hence \(C\) is nonempty. Moreover, \(C\) is stable under conjugation in \(G_E(M)\), since \(\det(I-A)\) depends only on the conjugacy class of \(A\). Now define the number fields
\[
K\coloneqq \Q(E[M])
\quad\text{and}\quad
\widetilde K\coloneqq K\!\left(\sqrt{\rad(2N_EM)}\right).
\]
Applying Theorem~\ref{ECDTA} with $m = \rad(2 N_E M)$, we find that there exists a prime
\(p\nmid 2N_EM\) such that
\begin{equation}\label{E:inCz}
\rho_{E,M}(\Frob_p)\in C
\end{equation}
and
\begin{equation}\label{eq:chebotarev-bound}
p\ll \bigl(\log d_{\widetilde K}+[\widetilde K:\Q]\bigr)^2,
\end{equation}
with an absolute implied constant.

Since $p \nmid M N_E$ and \eqref{E:inCz} holds, \(p\) is a prime of good reduction and no prime \(\ell\leq n\) divides
\(\#E_p(\F_p)\). Because \(n\geq 3\), the prime \(p\) is at least \(5\), and the Hasse bound gives
\(\#E_p(\F_p)>1\). Thus the least prime divisor \(c_E(p)\) is defined, and we have
\[
c_E(p)>n.
\]

It remains to bound the right-hand side of \eqref{eq:chebotarev-bound}. We first have
\begin{equation} \label{E:DegBd}
    [\widetilde K:\Q] \leq 2[K:\Q]=2|G_E(M)|\le 2|\GL_2(\Z/M\Z)|\le 2M^4.
\end{equation}
Moreover, \(K/\Q\) is ramified only at primes dividing \(N_EM\), so
\(\widetilde K/\Q\) is ramified only at primes dividing \(2N_EM\). Thus
\( \rad(d_{\widetilde K})\) divides \(\rad(2N_EM). \)
By \cite[Proposition 6]{MR644559},
\[
\log d_{\widetilde K}
\le
\bigl([\widetilde K:\Q]-1\bigr)\log \rad(d_{\widetilde K})
+
[\widetilde K:\Q]\log[\widetilde K:\Q].
\]
Using \eqref{E:DegBd} and allowing the implied constant to depend on $E$, we obtain
\begin{equation} \label{E:LogDiscBd}
\log d_{\widetilde K}\ll_E M^4\log M.
\end{equation}
Combining \eqref{eq:chebotarev-bound}, \eqref{E:DegBd}, and \eqref{E:LogDiscBd}, we find that
\begin{equation} \label{E:pBd} p \ll_E M^8 (\log M)^2. \end{equation}

The prime number theorem, in the equivalent form recalled in \cite[Theorem 4.4]{MR434929} gives that $\log M = \sum_{\ell \leq n} \log \ell$ is asymptotic to $n$ as $n \to \infty$. Taking natural logarithms in \eqref{E:pBd} therefore gives
\[
\log p \leq 8 \log M + 2 \log \log M + O_E(1) = (8 + o(1)) n
\]
as $n \to \infty$. Thus, for every fixed $0 < \kappa < 1/8$ and all sufficiently large $n$, we have 
\( n > \kappa \log p. \)
Since $c_E(p) > n$, it follows that 
\[
c_E(p) > \kappa \log p.
\]

As \(n\) may be chosen arbitrarily large, this construction gives primes \(p\) of good reduction with \(c_E(p)>n\) for every \(n\geq 3\). A fixed prime \(p\) can satisfy \(c_E(p)>n\) for only finitely many values of \(n\). Hence the set of primes obtained in this way is infinite, completing the proof.
\end{proof}

\bibliographystyle{amsalpha}
\bibliography{References}

\end{document}